\documentclass[a4paper,10pt]{article}

\usepackage{amsmath,amssymb,amsfonts,mathrsfs,amscd,amsthm,latexsym}
\usepackage{dina42e}
\usepackage{xcolor}				

\usepackage{makeidx}
\makeindex

\newtheorem{thm}{Theorem}[section]
\newtheorem{corol}[thm]{Corollary}
\newtheorem{lemma}[thm]{Lemma}
\newtheorem{prop}[thm]{Proposition}
\newtheorem{claim}[thm]{Claim}
\newtheorem{defin}[thm]{Definition}

\newtheorem{note}[thm]{Remark}

\title{Nonsmooth Pseudodifferential Boundary Value Problems on Manifolds}
\author{Helmut Abels\footnote{Fakult\"at f\"ur Mathematik,  
Universit\"at Regensburg,
93040 Regensburg,
Germany, e-mail: {\sf helmut.abels@mathematik.uni-regensburg.de}}~ and Carolina Neira Jim\'enez\footnote{Departamento de Matem\'aticas,
Universidad Nacional de Colombia,
Carrera 30 \# 45-03,
Bogot\'a, Colombia, e-mail: {\sf
cneiraj@unal.edu.co}}}

\newcommand{\R}{{\mathbb R}}
\newcommand{\Rn}{{\mathbb{R}^n}}
\newcommand{\C}{{\mathbb C}}
\newcommand{\N}{{\mathbb N}}
\newcommand{\Z}{{\mathbb Z}}

\DeclareMathOperator{\Op}{Op}
\DeclareMathOperator{\supp}{supp}
\DeclareMathOperator{\dist}{dist}

\newcommand{\dbar}{d{\hskip-1pt\bar{}}\hskip1pt}

\newcommand{\twobytwo}[4]{\begin{pmatrix}#1 & #2 \\ #3 & #4
                            \end{pmatrix}}

\begin{document}

\maketitle

\begin{abstract}
We study pseudodifferential boundary value problems in the context of the Boutet de Monvel calculus or Green operators, with nonsmooth coefficients on smooth compact manifolds with boundary. In order to have a definition that is independent of the choice of (smooth) coordinates, we  prove that nonsmooth Green operators are invariant under smooth coordinate transformations. 
\end{abstract}


\noindent{\bf AMS-Classification:} 35 S 15, 35 J 55

\section{Introduction}

In \cite{BoutetDeMonvel} L.\ Boutet de Monvel introduced a class of operators to model differential and pseudodifferential boundary value problems, which is closed under composition and adjoints and can be used to construct parametrices of elliptic operators. This was the basis for many other pseudodifferential calculi for boundary value problems,  cf.\ e.g.\ Grubb \cite{FunctionalCalculus}. Moreover, it  was used in index theory, cf.\  \cite{BoutetDeMonvel}, Rempel and Schulze \cite{RempelSchulze}, in the theory of Stokes equations, cf. Grubb and Solonnikov~\cite{MathScand,GrubbNStLowRegularity},  in geometrical problems as trace expansions, cf.\ e.g.\ Grubb and Schrohe~\cite{TraceExpansions}, and others, cf.\ \cite{FunctionalCalculus}. 

In most pseudodifferential calculi smoothness of the symbols with respect to all variables is assumed. But in applications to nonlinear partial differential equations, operators with nonsmooth symbols with respect to the space variable $x$ occur naturally since the symbol depends on the solution of the PDE itself, which has a~priori limited smoothness. 
In the case of pseudodifferential operators on $\R^n$ symbol classes with nonsmooth $x$-dependence were first introduced by Kumano-Go and Nagase~\cite{KumanoGoNagase}. Properties of these operators were studied e.g.\ by Marschall~\cite{MarschallZAA}, Witt~\cite{Witt}, Taylor~\cite{TaylorNonlinearPDE, ToolsForPDE}, Barraza-Mart{\'{\i}}nez, Denk and Hern{\'a}ndez-Monz{\'o}n~\cite{DenkEtAlNonsmoothPsDOs}, and A.\ and Pfeuffer~\cite{AbelsPfeufferSpectralInvariance,AbelsPfeufferCharacterization, AbelsPfeufferFredholm}.

In the case of a half-space $\overline{\R^n_+}$ the calculus for pseudodifferential boundary value problems was extended to nonsmooth symbols by the first author in \cite{NonsmoothGreen, HInftyInLayer}. The calculus was applied e.g.\ in A.\ and Terasawa~\cite{BIPVariableViscosity},  A., Grubb and Wood~\cite{AbelsGrubbWood}, and Grubb~\cite{GrubbSpectralAsymptotics} to study the Stokes equations, elliptic boundary value problems in domains with low boundary regularity, and spectral asymptotics for selfadjoint nonnegative singular Green operators.
In the case of nonsmooth $x$-dependence the usual algebraic properties of the associated operator classes are no longer true, e.g.\ the class is no longer closed under composition and the usual asymptotic expansions no longer hold true. But there are relaxed versions of these statements up to lower order operators. E.g.\ under suitable smoothness assumptions the composition of two nonsmooth pseudodifferential operators is a pseudodifferential operator up to an operator of lower order and a finite piece of the usual asymptotic expansion holds true, cf.\ \cite[Theorem 1.2]{NonsmoothGreen}.

In the present contribution we want to complement the results in \cite{NonsmoothGreen} , where only the case that the domain is a half-space is treated, and the results in the master thesis \cite{koeppl:PDOSnonsmooth}, where pseudodifferential operators with nonsmooth coefficients on manifolds were studied. We will show that the calculus is closed under smooth coordinate changes which preserve the boundary, and define nonsmooth Green operators on smooth manifolds. Of course one could also consider nonsmooth Green operators on manifolds with limited smoothness. But on one hand this would make the calculus very technical and complicated, e.g.\ when one considers operators of non-zero order the smoothness assumptions for the domain and target manifold would need to be different to get sharp results. On the other hand in applications one often parametrize a nonsmooth submanifold (or domain) in $\R^n$ by a smooth ``reference'' manifold (or domain) close by, cf.\ e.g.\ \cite[Section~2.3]{PruessSimonett}. By a simple pull-back equations on the nonsmooth manifold carry over to equations on the smooth reference manifold (with nonsmooth coefficients). In this way nonsmoothness can be transferred from the manifolds to the symbols of the operators. That is the reason why we restrict ourselves to nonsmooth Green operators on smooth manifolds. 
The article is organized as follows: In Section~\ref{sec:prelim} we give the preliminaries needed to show that the calculus for nonsmooth pseudodifferential boundary value problems is closed under coordinate changes. 
Section~\ref{section:change of variables} is devoted to the proof of the invariance under a suitable coordinate change for nonsmooth pseudodifferential operators on the Euclidean space. The strategy corresponds to ``freezing the coefficients" of the symbols of the operators and the idea is to express the coordinate changes of a symbol as a composition of some appropriate continuous maps. In Section~\ref{section:Truncation of nonsmooth pdos} we study the kernel of nonsmooth pseudodifferential operators after composition with multiplication operators, which appear in the definition of nonsmooth pseudodifferential operators on manifolds.
In Section~\ref{sec:Green} we recall the definition of nonsmooth Green operators: nonsmooth Poisson operators, nonsmooth trace operators and nonsmooth singular Green operators. After this, in Section~\ref{section:change of variables Green} we show the invariance under a coordinate change for these operators on the Euclidean space, again by the technique of freezing the coefficients. In Section~\ref{sec:Localization} we study the kernel of nonsmooth Green operators after composition with multiplication operators.
In Section~\ref{section:transmission condition} we show that the nonsmooth transmission condition is preserved under a smooth coordinate change. Finally Section~\ref{sec:Manifolds} is devoted to the definition of nonsmooth Green operators on smooth manifolds. 

\smallskip

\noindent
\textbf{Acknowlegdement:} The second author gratefully acknowledges support through DFG, GRK 1692 ``Curvature, Cycles and Cohomology'' during parts of the work.

\section{Preliminaries}\label{sec:prelim}

Along the document $C$ will denote a generic constant.

\subsection{H\"older Spaces}

By $F$ we will denote a Fr\'echet space equipped with a family of seminorms $\{|\cdot|_i\}_{i\in\N}$. The set $L^\infty(\R^n;F)$ consists of strongly measurable functions $f\colon \R^n\to F$ such that $\underset{z\in \R^n}{{\rm ess}\sup}\{|f(z)|_i\}<\infty$ for all $i\in\N$. Then, $L^\infty(\R^n;F)$\index{$L^\infty(\R^n;F)$} is equipped with the countable set of seminorms given by 
\[
 \|f\|^i_{L^\infty(\R^n;F)}:=\underset{z\in \R^n}{{\rm ess}\sup}\{|f(z)|_i\} \text{\ \ for all }i\in\N.
\]

\begin{defin}[{\cite[A.1]{TaylorNonlinearPDE}}]
\label{Def:Holder space}
 Let $\tau>0$. The H\"older space of $F$--valued functions on $\R^n$ of degree $\tau$, $C^\tau(\R^n;F)$\index{$C^\tau(\R^n;F)$}, consists of functions $f:\R^n\to F$ with H\"older continuous derivatives $\partial_z^\beta f$ of degree $\tau-[\tau]$ for all $|\beta|\leq [\tau]$, equipped with the countable set of seminorms given for all $i\in\N$ by
\begin{equation}
\label{normCtau}
 \lVert f\lVert^i_{C^\tau(\R^n;F)}:=\sum_{|\beta|\leq[\tau]}\left\lVert\partial_z^\beta f\right\lVert^i_{L^\infty(\R^n;F)} + \sum_{|\beta|=[\tau]}\underset{z\not=y}{\sup}\dfrac{\left|\partial_z^\beta f(z) - \partial_z^\beta f(y)\right|_i}{|z-y|^{\tau-[\tau]}}.
\end{equation}
\end{defin}

We also use the notation $C^\tau_z(\R^n;F)$\index{$C^\tau_z(\R^n;F)$} to indicate the space of functions which are H\"older continuous with respect to the variable $z\in \R^n$, and the notation $ \lVert f\lVert^i_{C^\tau_z(\R^n;F)}$ if the norm of the function $f$ is taken with respect to $z$.

\begin{note}
\label{rk:norm Ctau x-ygeq1}
 In the definition of the seminorms $\lVert \cdot\lVert^i_{C^\tau(\R^n;F)}$ given in \eqref{normCtau}, in the second term on the right hand side it is enough to consider $z\not=y$ such that $|z-y|<1$. Indeed, if $z,y\in \R^n$ are such that $|z-y|\geq1$, then
\begin{align*}
 \left|\partial_z^\beta f(z) - \partial_z^\beta f(y)\right|_i
 &\leq2\left\lVert\partial_z^\beta f\right\lVert^i_{L^\infty(\R^n;F)}
\leq C|z-y|^{\tau-[\tau]}.
\end{align*}
\end{note}

\begin{note}
\label{partial f Ctau-1}
 Let $\tau\notin\N,\tau>1$. If $f\in C^\tau(\R^n;F)$, then  $\partial_{z_j}f\in C^{\tau-1}(\R^n;F)$ for all $j=1,\ldots,n$. Moreover, $\left\lVert\partial_{z_j}f\right\lVert^i_{C^{\tau-1}(\R^n;F)}\leq\left\lVert f\right\lVert^i_{C^\tau(\R^n;F)}$ for all $i\in\N$.
\end{note}

Finally, for an open set $U\subseteq \R^n$ we denote by $C^\infty (U;F)$ the set of all smooth $f\colon U\to F$, $C^\infty_b(U;F)$ the subset of all $f\in C^\infty(U;F)$ that are bounded and have bounded derivatives of arbitrary order (with respect to all seminorms), and $C^\infty_0(U;F)$ is the set of all $f\in C^\infty (U;F)$ with compact support. In the case $F=\C$ we simply write $C^\infty(U), C^\infty_b(U)$ and $C^\infty_0(U)$ for the corresponding spaces.
\subsection{Nonsmooth Pseudodifferential Operators}
Let $X$ be a Banach space, with norm $\|\cdot\|_X$. For $\xi\in\R^n$, we use the notation $\langle\xi\rangle:=(1+|\xi|^2)^{1/2}$.
\begin{defin}
[{\cite[Def.\ 3.1]{NonsmoothGreen} and \cite[Def.\ 3.1]{koeppl:PDOSnonsmooth}}]
\label{def:Ctau symbols}
 The symbol space $C^\tau S^m_{1,0}(\R^n\times \R^n;X)$\index{$C^\tau S^m_{1,0}(\R^n\times \R^n;X)$}, $\tau>0$, $m\in\R$, is the set of all functions $p:\R^n\times \R^n\to X$ that are smooth with respect to $\xi\in \R^n$ and are in $C^\tau$ with respect to $x\in \R^n$ satisfying for all $\alpha\in\N^{n}$ the following estimates
\begin{align}
 \left\lVert\partial_{\xi}^\alpha p(\cdot,\xi)\right\lVert_{C^\tau(\R^n;X)}&\leq C_{\alpha}\langle\xi\rangle^{m-|\alpha|} \quad \quad \text{ for all }\xi\in\R^n.\label{estimate2}
\end{align}
 We will often use the shorthand notation $C^\tau S^m_{1,0}$\index{$C^\tau S^m_{1,0}$} for $C^\tau S^m_{1,0}(\R^n\times \R^n;X)$. 
 We equip the space $C^\tau S^m_{1,0}$ with the following seminorms: For all $p\in C^\tau S^m_{1,0}$ and $i\in\N$,
\begin{equation}
 |p|_{C^\tau S^m_{1,0}}^{i}:=\underset{|\alpha|\leq i}{\max}\ \underset{\xi\in \R^n}{\sup}\left\lVert\partial_{\xi}^\alpha p(\cdot,\xi)\right\lVert_{C^\tau(\R^n;X)}\langle\xi\rangle^{-m+|\alpha|}.
\end{equation}
\end{defin}
We note that
\begin{equation*}
  S^m_{1,0}(\R^n\times \R^n;X)=\bigcap_{\tau>0} C^\tau S^m_{1,0}(\R^n\times \R^n;X),
\end{equation*}
where $S^m_{1,0}(\R^n\times \R^n;X)$\index{$S^m_{1,0}(\R^n\times \R^n;X)$} is the standard H\"ormander class of smooth pseudodifferential symbols in its $X$--valued variant. For all $p\in S^m_{1,0}$ and $i\in\N$
\[
 |p|_i^{(m)}:=\underset{|\alpha|,|\beta|\leq i}{\max}\ \underset{(x,\xi)\in \R^n\times \R^n}{\sup}\left\|\partial_{\xi}^\alpha\partial_x^\beta p(x,\xi)\right\|_X\langle\xi\rangle^{-m+|\alpha|}
\]
defines a sequence of seminorms, which yield the standard topology on $S^m_{1,0}(\R^n\times \R^n;X)$.

For the following let $X:=\mathcal{L}(X_0,X_1)$, where $X_0,X_1$ are two Banach spaces, let $\mathcal{S}(\R^n;X_0)$\index{$\mathcal{S}(\R^n;X_0)$} be the space of $X_0$--valued smooth rapidly decreasing functions on $\R^n$ (if $X_0=\C$, we use the notation $\mathcal{S}(\R^n)$), and let $p\in C^\tau S^m_{1,0}(\R^n\times \R^n;X)$. Then
\begin{align*}
  (\Op(p)u)(x)\equiv (p(x,D_x)u)(x)&:=\int_{\R^n}{e^{ix\cdot\xi}p(x, \xi)\widehat u(\xi)\,\dbar\xi} \qquad \text{for all }x\in\R^n, u \in \mathcal{S}(\R^n;X_0),
\end{align*}
defines a bounded linear operator $p(x,D_x)\colon \mathcal{S}(\R^n;X_0)\to C^\tau(\R^n;X_1)$\index{$p(x,D_x)$}, which is the {\it pseudodifferential operator} associated to the symbol $p$.
Here $\widehat u(\xi):=\int_{\R^n}{e^{-iy\cdot\xi}u(y)\,dy}$ is the Fourier transform of $u$ and $\dbar\xi:=(2\pi)^{-n}d\xi$.

If $p\in C^\tau S^m_{1,0}$, then $p(x,\xi)$ is smooth in $\xi$. Hence we obtain as in the smooth case
\begin{align}
 p(x,D_x)f(x)&=\int_{\R^n}e^{ix\cdot\xi}p(x,\xi)\hat{f}(\xi)\,\dbar\xi 
 ={\rm Os} - \iint e^{i(x-y)\cdot\xi}p(x,\xi)f(y)\,dy\,\dbar\xi \label{eq:oscillatory integral}
\end{align}
for all $x\in\R^n$ and $f\in \mathcal{S}(\R^n;X_0)$, since $p(x,\xi)f(y)$ belongs to the standard space of amplitudes ${\mathcal{A}_{0,0}^m}$\index{${\mathcal{A}_0^m}(\R^n\times\R^n)$} with respect to $(y,\xi)\in\R^n\times\R^n$, cf.\ e.g.\ \cite[Chapter 1, \S 6]{KumanoGo}. We note that $x$ is only a fixed parameter for the oscillatory integral. Therefore, the smoothness with respect to $x$ is irrelevant. We recall that ${\mathcal{A}_{0,0}^m}$ consists of all smooth $a\colon \R^n\times \R^n \to X$ such that for all $\alpha,\beta \in \N_0^n$ there is some constant $C_{\alpha,\beta}>0$ such that
\begin{equation*}
  \|\partial_y^\alpha\partial_\eta^\beta a(y,\eta)\|_X \leq C_{\alpha,\beta} (1+|\xi|)^{m}\qquad \text{for all }y,\eta\in\R^n.
\end{equation*}
In the case $X=\C$ the set and the associated oscillatory integrals are discussed in \cite[Chapter 1, \S 6]{KumanoGo}. The definition and results carry over to the $X$-valued setting in a straight forward manner.

\subsection{Coordinate Changes in the Smooth Case}

\begin{defin}[{\cite[A.5]{FunctionalCalculus}}]
\label{def:bounded smooth diffeo}
Let $U$ be either $\R^n$ or $\overline{\R_+^n}:=\{x\in\R^n:x_n\geq0\}$. A smooth diffeomorphism $\kappa:U\to U$ is said to be a \emph{bounded smooth diffeomorphism} if all derivatives of $\kappa$ and $\kappa^{-1}$ are bounded.
\end{defin}

If $\kappa$ is a bounded smooth diffeomorphism, then by the mean value theorem in several variables applied to $\kappa^{-1}$ (see \cite[Lemma 2.3]{koeppl:PDOSnonsmooth}), for all $x,y\in U$ the matrix
\[
 M_{\kappa^{-1}}(x,y):=\int_0^1\mathcal{D}(\kappa^{-1})(x+t(y-x))\,dt,
\]
satisfies
\begin{equation}
\label{eq:regular diffeo matrix}
 \kappa^{-1}(x)-\kappa^{-1}(y)=M_{\kappa^{-1}}(x,y)\big(x-y\big).
\end{equation}
By the same argument applied to $\kappa$ instead of $\kappa^{-1}$ and \eqref{eq:regular diffeo matrix}, we obtain that there exist constants $C_1,\,C_2>0$ such that (see \cite[(A.61)]{FunctionalCalculus} and \cite[(4.8)]{koeppl:PDOSnonsmooth})
\begin{equation}
 C_1|x-y|\leq \left|\kappa^{-1}(x)-\kappa^{-1}(y)\right|\leq C_2|x-y| \qquad \text{ for all }x,y\in U.
\end{equation}

Now we can state the invariance of pseudodifferential operators under coordinate changes in the following way:
\begin{thm}
\label{thm:change of variables smooth pdos 2}
Let $\kappa:\R^n\to \R^n$ be a bounded smooth diffeomorphism and let $A=p(x,D_x)$, where $p\in S^m_{1,0}(\Rn\times\R^n)$ for some $m\in\R$. Then, the operator $A_\kappa$ given by $A_\kappa(u):=[A(u\circ\kappa)]\circ\kappa^{-1}$ for all $u\in C_0^\infty(\R^n)${} is a pseudodifferential operator of order $m$ on $\R^n$. Moreover, if we denote by $p_\kappa$\index{$p_\kappa$} its symbol, the mapping 
$$S^m_{1,0}(\Rn\times \Rn)\ni p\mapsto p_\kappa \in S^m_{1,0}(\Rn\times \Rn)$$ 
is continuous.
\end{thm}
The first part follows e.g.\ from \cite[Chapter 2, Theorem~6.3]{KumanoGo}. Analyzing the proof one easily sees that the mapping is bounded.
\begin{note}
  The result and proof of Theorem~\ref{thm:change of variables smooth pdos 2} directly carries over to the case that $p\in S^m_{1,0}(\Rn\times \Rn;X)$ with $X=\mathcal{L}(X_0,X_1)$ as before.
\end{note}

\section{Coordinate Changes for Nonsmooth Pseudodifferential Operators}
\label{section:change of variables}

Now we generalize Theorem \ref{thm:change of variables smooth pdos 2} to the case of nonsmooth pseudodifferential operators. To this end we use bounded smooth diffeomorphisms as introduced in Definition \ref{def:bounded smooth diffeo}.

\begin{thm}
 \label{theorem1}
 Let $\kappa\colon \R^n\to \R^n$ be a bounded smooth diffeomorphism and let $p\in C^\tau S^m_{1,0}(\R^n\times \R^n;\mathcal{L}(X_0,X_1))$, where $m\in\R$, $\tau>0$, $\tau\notin\N$, and $X_0,X_1$ are Banach spaces. Then there is a $\widetilde{p}\in C^\tau S^m_{1,0}(\R^n\times \R^n;\mathcal{L}(X_0,X_1))$\index{$\widetilde{p}$} such that
\begin{equation}
\label{eq:change of variables}
 \widetilde{p}(x,D_x)f(x)=\kappa^{-1,*}\circ p(x,D_x)\circ\kappa^*f(x)\quad \text{ for all }x\in\R^n, f\in\mathcal{S}(\R^n;X_0),
\end{equation}
where $\kappa^*f:=f\circ\kappa$.
\end{thm}
\begin{note}\label{rem:tauN} In the case $\tau\in \N$ the statement of the theorem holds true if one replaces $C^\tau$ by $C^{\tau-1,1}$, where $C^{k,1}(\Rn;F)$ consists of all functions $f\in C^k(\Rn;F)$ such that $\partial_x^\alpha f\colon \Rn\to F$ is globally Lipschitz continuous for all $|\alpha|=k$. The following proofs can be easily carried over to that case.
\end{note}
\begin{proof}{\it (Theorem~\ref{theorem1})}
 For the following let $X:=\mathcal{L}(X_0,X_1)$ and  be aware of the two different sets:
\begin{enumerate}
 \item $C^\tau S^m_{1,0}$, the set of functions $p\colon \R^n\times \R^n\to X$ from Definition \ref{def:Ctau symbols}.
 \item $C^\tau (\R^n;S^m_{1,0}(\R^n\times \R^n;X))$, that we will denote by $C^\tau (\R^n;S^m_{1,0})$, the set of H\"older continuous functions $f:\R^n\to S^m_{1,0}$, such that for every $i\in\N$, the value of the corresponding seminorm $\lVert\cdot\lVert^i_{C^\tau (\R^n;S^m_{1,0})}$ (defined in \eqref{normCtau}) applied to $f$ is finite.
\end{enumerate}

The idea of the proof is to see the operator defined by \eqref{eq:change of variables} as a composition of the following continuous maps:
\begin{alignat*}{2}
	\Phi_1&\colon C^\tau S^m_{1,0}\to C^\tau(\R^n;S^m_{1,0})&&\colon p\mapsto (z\mapsto q_z),\\
  	\Phi_2&\colon C^\tau(\R^n;S^m_{1,0})\to C^\tau(\R^n;S^m_{1,0})&&\colon (z\mapsto q_z)\mapsto (z\mapsto q_{\kappa^{-1}(z)})\\
	\Phi_3&\colon C^\tau(\R^n;S^m_{1,0})\to C^\tau(\R^n;S^m_{1,0})&&\colon (z\mapsto q_z)\mapsto (z\mapsto Tq_z)\\
        \Phi_4&\colon C^\tau(\R^n;S^m_{1,0})\to C^\tau S^m_{1,0}&&\colon (z\mapsto q_z)\mapsto q_{z}|_{z=x}        
\end{alignat*}
where $q_z(x,\xi):=p(z,\xi)$ for all $z,x,\xi\in \R^n$,
$$T\colon S^m_{1,0}(\R^n\times\R^n;\mathcal{L}(X_0,X_1))\to S^m_{1,0}(\R^n\times\R^n;\mathcal{L}(X_0,X_1))\colon p\mapsto p_\kappa$$ 
represents the change of coordinates map as in Theorem~\ref{thm:change of variables smooth pdos 2}, $p_\kappa$ represents the symbol of the operator defined in that theorem, and $q_z|_{z=x}(x,\xi)= q_x(x,\xi)$ for all $x,\xi\in\R^n$.

The map $\Phi_1$ corresponds to a ``freezing of coefficients'' since it takes a nonsmooth $p$ and maps it to a family of smooth symbols $q_z=p(z,\cdot)$ depending on the spatial variable $z$. Moreover, $\Phi_2$ treats the coordinate transformation only in the ``frozen'' spatial variable $z\in\Rn$ and  $\Phi_3$ describes the coordinate change for the smooth symbol $q_z$ for each $z$. Finally, $\Phi_4$ corresponds to ``unfreezing'' the coefficients by evaluating ``$z=x$''.

Next we will prove that these maps are well--defined and continuous. Then the symbol of the operator given by \eqref{eq:change of variables} can be written as
\[
 \widetilde{p}(x,\xi)=\Phi_4(\Phi_3(\Phi_2(\Phi_1(p))))(x,\xi),
\]
for all $(x,\xi)\in \R^n\times \R^n$ and we conclude the statement of the theorem.

The continuity of $\Phi_2$ and $\Phi_3$ are obvious since these mappings act only with respect to $z\in \R^n$, $q_z\in S^m_{1,0}$, respectively, and the latter actions are well defined and continuous due to Theorem~\ref{thm:change of variables smooth pdos 2} because $\kappa$ is a bounded smooth diffeomorphism. 

\begin{lemma}
\label{lemma:Phi1}
Let $m\in\R$, $\tau>0$, $\tau\not\in \N$. Then the map 
 \begin{align*}
  \Phi_1:C^\tau S^m_{1,0}&\to C^\tau(\R^n;S^m_{1,0})\colon p\mapsto q,
 \end{align*}
 where $q_z(x,\xi):=p(z,\xi)$ for all $z,x,\xi\in \R^n$, is well--defined and continuous.
\end{lemma}

\begin{proof}
 For $z\in \R^n$ fixed, the function $q_{z}$ belongs to $S^m_{1,0}$. In fact, $q_{z}$ has constant coefficients and therefore it is smooth with respect to $x$. Since $p\in C^\tau S^m_{1,0}$, for any pair of multiindices $\alpha,\beta\in\N^{n}$ there exists a constant $C_{\alpha,\beta}$ such that
 \[
  \left\|\partial_{\xi}^\alpha\partial_x^\beta q_{z}(x,\xi)\right\|_X=\left\|\partial_{\xi}^\alpha\partial_x^\beta \left(p(z,\xi)\right)\right\|_X\leq C_{\alpha,\beta}\langle\xi\rangle^{m-|\alpha|}
 \]
 which implies that $q_{z}\in S^m_{1,0}$ for all $z\in \R^n$.\\

By definition of $q_z$, if $p\in C^\tau$ with respect to $x$, then $q\in C^\tau$ with respect to $z$. Indeed, for all $i\in\N$ and for all $\delta\in\N^{n}$ such that $|\delta|\leq[\tau]$
\begin{align}
 \underset{z\in \R^n}{\sup}\left|\partial_z^\delta q_z\right|^{(m)}_i
 &=\underset{z\in \R^n}{\sup}\ \underset{|\alpha|,|\beta|\leq i}{\max}\ \underset{(x,\xi)\in \R^n\times \R^n}{\sup}\left\|\partial_\xi^\alpha\partial_x^\beta\partial_z^\delta q_z(x,\xi)\right\|_X\langle\xi\rangle^{-m+|\alpha|} \notag \\
 &=\underset{z\in \R^n}{\sup}\ \underset{|\alpha|\leq i}{\max}\ \underset{\xi\in \R^n}{\sup}\left\|\partial_\xi^\alpha\partial_z^\delta \left(p(z,\xi)\right)\right\|_X\langle\xi\rangle^{-m+|\alpha|} \leq C\,|p|_{C^\tau S^m_{1,0}}^{i}. \label{norm qz1}
\end{align}
Moreover for $z^0,z^1\in \R^n$ and for all $\delta\in\N^{n}$ such that $|\delta|=[\tau]$,
\begin{align}
 &\left|\partial_z^\delta q_{z^0}-\partial_z^\delta q_{z^1}\right|_i^{(m)} \notag\\
 &\ =\underset{|\alpha|,|\beta|\leq i}{\max}\ \underset{(x,\xi)\in \R^n\times \R^n}{\sup}\left\|\partial_\xi^\alpha\partial_x^\beta\partial_z^\delta{q_{z^0}}(x,\xi)-\partial_\xi^\alpha\partial_x^\beta\partial_z^\delta{q_{z^1}}(x,\xi)\right\|_X\langle\xi\rangle^{-m+|\alpha|} \notag \\
 &\ =\underset{|\alpha|\leq i}{\max}\ \underset{\xi\in \R^n}{\sup}\left\|\partial_\xi^\alpha\partial_z^\delta \left({p}(z^0,\xi)\right)-\partial_\xi^\alpha\partial_z^\delta \left(p(z^1,\xi)\right)\right\|_X\langle\xi\rangle^{-m+|\alpha|} \notag \\
&\ \leq C\ |z^0-z^1|^{\tau-[\tau]}\,|p|_{C^\tau S^m_{1,0}}^{i}. \notag 
\end{align}
Hence, for all $i\in\N$ and for all $\delta\in\N^{n}$ such that $|\delta|=[\tau]$
\begin{equation}
 \label{norm qz2}
 \underset{\substack{z^0, z^1\in \R^n \\ z^0\not=z^1}}{\sup}\dfrac{\left|\partial_z^\delta q_{z^0}-\partial_z^\delta q_{z^1}\right|_i^{(m)}}{|z^0-z^1|^{\tau-[\tau]}}\leq C\ |p|_{C^\tau S^m_{1,0}}^{i}.
\end{equation}
Both \eqref{norm qz1} and \eqref{norm qz2} imply that for all $i\in\N$ there exists a constant $C$ such that 
\[
 \|q\|_{C^\tau(\R^n;S^m_{1,0})}^{i}\leq C\ |p|_{C^\tau S^m_{1,0}}^{i}.
\]
\end{proof}

\begin{lemma}
\label{lemma:Psi}
 The map $\Phi_4\colon C^\tau(\R^n;S^m_{1,0})\to C^\tau S^m_{1,0}$ is well--defined and continuous.
\end{lemma}

\begin{proof}
First of all, let $\widetilde{q}(x,\xi):= q_z(x,\xi)|_{z=x}$. Then  we have
\begin{align}\nonumber
 &\underset{|\alpha|\leq i}{\max}\ \underset{\xi\in \R^n}{\sup}\ \left\|\partial_\xi^\alpha\widetilde{q}(\cdot,\xi)\right\|_{L^\infty(\R^n;X)}\langle\xi\rangle^{-m+|\alpha|} \\\nonumber
 &\ \leq\underset{|\alpha|\leq i}{\max}\ \underset{(x,\xi)\in \R^n\times \R^n}{\sup}\ \left\|\left.\partial_\xi^\alpha q_z(x,\xi)\right|_{z=x}\right\|_X\langle\xi\rangle^{-m+|\alpha|}\\\nonumber
 &\ \leq C\,\underset{z\in \R^n}{\sup}\ \underset{|\alpha|\leq i}{\max}\underset{(x,\xi)\in \R^n\times \R^n}{\sup}\left\|\partial_\xi^\alpha\partial_x^\beta q_z(x,\xi)\right\|_X\langle\xi\rangle^{-m+|\alpha|}\\\label{estimate norm infty tilde p}
 & \ =C\,\left\|q_\bullet\right\|^i_{L^\infty(\R^n;S^m_{1,0})}.
\end{align}

\begin{claim}
\label{Case tau in (0,1)}
 Let $q\in C^\tau(\R^n;S^m_{1,0})$ for $\tau\in(0,1)$. Then $\widetilde{q}\in C^\tau S^m_{1,0}$.
\end{claim}

\begin{proof}
Let $\alpha\in\N^{n}$ be any multiindex. By the triangle inequality we have
\begin{align}
 &\left\|\partial_{\xi}^\alpha\widetilde{q}(x^0,\xi)-\partial_{\xi}^\alpha\widetilde{q}(x^1,\xi)\right\|_X\langle\xi\rangle^{-m+|\alpha|}\notag \\
 &\leq \left\|\partial_{\xi}^\alpha{q_{x^0}}(x^0,\xi)-\partial_{\xi}^\alpha{q_{x^1}}(x^0,\xi)\right\|_X\langle\xi\rangle^{-m+|\alpha|}+ \notag \\
 &\quad +\left\|\partial_{\xi}^\alpha{q_{x^1}}(x^0,\xi)-\partial_{\xi}^\alpha{q_{x^1}}(x^1,\xi)\right\|_X\langle\xi\rangle^{-m+|\alpha|}. \label{triangleinequality2}
\end{align}
Applying the mean value theorem in several variables for all $x^0,x^1\in \R^n$ such that $|x^0-x^1|<1$ we obtain 
\begin{align}\nonumber
 &\left\|\partial_{\xi}^\alpha{q_{x^1}}(x^0,\xi)-\partial_{\xi}^\alpha{q_{x^1}}(x^1,\xi)\right\|_X\langle\xi\rangle^{-m+|\alpha|} \\\nonumber
 &\ \leq \underset{0\leq c\leq1}{\sup}\left\|\nabla_x\partial_{\xi}^\alpha{q_{x^1}}((1-c)x^0+cx^1,\xi)\right\|_X\langle\xi\rangle^{-m+|\alpha|}|x^0-x^1| \\
 &\ \leq \left|{q_{x^1}}\right|_{|\alpha|+1}^{(m)}|x^0-x^1| \leq \left\|q\right\|^{|\alpha|+1}_{C^\tau(\R^n;S^m_{1,0})}|x^0-x^1|^\tau,\label{Second term rhs}
\end{align}
where $\nabla_x f$ represents the gradient of the function $f$ with respect to $x$.\\

By definition of the seminorm $|\cdot|_{|\alpha|}^{(m)}$
there exists $C\in\R$ such that 
\begin{align}
 \left\|\partial_{\xi}^\alpha{q_{x^0}}(x^0,\xi)-\partial_{\xi}^\alpha{q_{x^1}}(x^0,\xi)\right\|_X\langle\xi\rangle^{-m+|\alpha|}
 &\leq \left|q_{x^0}-q_{x^1}\right|^{(m)}_{|\alpha|} \notag \\
 & \leq C\left\|q\right\|^{|\alpha|}_{C^\tau(\R^n;S^m_{1,0})}|x^0-x^1|^\tau. \label{First term rhs}
\end{align}
An application of \eqref{Second term rhs} and \eqref{First term rhs} in \eqref {triangleinequality2} gives us the existence of a constant $C$ such that
\begin{equation*}
 \left\|\partial_{\xi}^\alpha\widetilde{q}(x^0,\xi)-\partial_{\xi}^\alpha\widetilde{q}(x^1,\xi)\right\|_X\langle\xi\rangle^{-m+|\alpha|}\leq C\left\|q\right\|^{|\alpha|+1}_{C^\tau(\R^n;S^m_{1,0})}|x^0-x^1|^\tau
\end{equation*}
for all $x^0,x^1\in \R^n$ such that $|x^0-x^1|<1$. \\ 
Therefore by Remark \ref{rk:norm Ctau x-ygeq1} there exists a constant $C$ such that
\[
 \underset{x^0\not=x^1}{\sup}\dfrac{\left\|\partial_{\xi}^\alpha\widetilde{q}(x^0,\xi)-\partial_{\xi}^\alpha\widetilde{q}(x^1,\xi)\right\|_X}{|x^0-x^1|^{\tau}}\langle\xi\rangle^{-m+|\alpha|}\leq C\left\|q\right\|^{|\alpha|+1}_{C^\tau(\R^n;S^m_{1,0})}.
\]
Thus, by \eqref{estimate norm infty tilde p} for all $\alpha\in\N^{n}$, there exists a constant $C$ such that for all $\xi\in \R^n$
\[
 \left\lVert \partial_{\xi}^\alpha\widetilde{q}(\cdot,\xi)\right\lVert_{C^\tau(\R^n;X)}\langle\xi\rangle^{-m+|\alpha|}\leq C\left\|q\right\|^{|\alpha|+1}_{C^\tau(\R^n;S^m_{1,0})}.
\]
Hence $\widetilde{q}\in C^\tau S^m_{1,0}$. Moreover, this shows that for all $i\in\N$ there exists some constant $C_i\in\R$ such that
\[
 \left\|\widetilde{q}\right\|^{i}_{C^\tau S^m_{1,0}}\leq C_i\left\|q\right\|^{i+1}_{C^\tau(\R^n;S^m_{1,0})}.
\]
\end{proof}

\begin{claim}
\label{Case tau not in N}
 Let $q\in C^\tau(\R^n;S^m_{1,0})$ for $\tau>0$, $\tau\notin\N$. Then $\widetilde{q}\in C^\tau S^m_{1,0}$.
\end{claim}

\begin{proof}
Let $\alpha\in\N^{n}$. By \eqref{estimate norm infty tilde p},  we already have the estimate $\left\lVert\partial_{\xi}^\alpha\partial_{x}^\beta\widetilde{q}(\cdot,\xi)\right\lVert_{L^\infty(\R^n;X)}\le C_{\alpha,\beta}\langle\xi\rangle^{m-|\alpha|}$ in the case $\beta=0$. Therefore we need to prove that $\partial_{\xi}^\alpha\widetilde{q}$ belongs to $C^{[\tau]}(\R^n;X)$ with respect to $x$, that for all $\beta\in\N^{n}$ with $|\beta|=[\tau]$, $\partial_{\xi}^\alpha\partial_{x}^\beta\widetilde{q}$ belongs to $C^{\tau-[\tau]}(\R^n;X)$ with respect to $x$, and then conclude that $\widetilde{q}$ satisfies estimates as in \eqref{estimate2}. This can be proved by mathematical induction with respect to $[\tau]\in \N$.\\ 

The case $[\tau]=0$ was proved in Claim~\ref{Case tau in (0,1)} before. So, let $l\in \N$, and assume that for every $\tau \notin \N$ with $0\leq [\tau]\leq l$ we have that ${q}\in C^\tau(\R^n;S^m_{1,0})$ implies that $\widetilde{q}\in C^\tau S^m_{1,0}$. Let $\tau\notin\N$ be such that $[\tau]=l+1$.\\

\noindent Since ${q}\in C^\tau(\R^n;S^m_{1,0})$, for all $\gamma\in\N^{n}$ with $1\leq|\gamma|\leq[\tau]=l+1$, $\partial_z^\gamma{q}\in C^{\tau-|\gamma|}(\R^n;S^m_{1,0})$, and then by induction hypothesis $\partial_z^\gamma{q}_z(x,\xi)|_{z=x}\in C^{\tau-|\gamma|} S^m_{1,0}$. \\

\noindent For all $\beta\in\N^{n}$, $1\leq|\beta|\leq[\tau]$, there exist $j\in\{1,\ldots,n\}$ and $\delta\in\N^{n}$ such that $\partial_{x}^\beta=\partial_{x}^\delta\partial_{x_j}$. Then by the chain rule we have
\begin{equation}
 \partial_{x}^\beta\widetilde{q}(x,\xi)=\partial_{x}^\delta\partial_{x_j}\widetilde{q}(x,\xi)
 =\partial_{x}^\delta\left(\left.\partial_{z_j} q_{x}(x,\xi)\right|_{z=x}\right)+\partial_{x}^\delta\left(\left.\partial_{x_j}{q_{z}}(x,\xi)\right|_{z=x}\right) \label{eq:derivative tilde p2}.
\end{equation}

\noindent So $\widetilde{q}(x,\xi)$ is also differentiable with respect to $x$ up to order $[\tau]$. By induction hypothesis, the facts that $[\tau-1]\leq l$ and $\partial_{z_j}{q}\in C^{\tau-1}(\R^n;S^m_{1,0})$, imply that $\left.\partial_{z_j}q_z(x,\xi)\right|_{z=x}\in C^{\tau-1}S^m_{1,0}$. Also the map $x\mapsto\left.\partial_{x_j}{q_{z}}(x,\xi)\right|_{z=x}$ lies in $C^{\tau}_x$. Since $C^{\tau}_x \hookrightarrow C^{\tau-1}_x$ and $\left.\partial_{x_j}{q_{z}}(x,\xi)\right|_{z=x}\in C^{\tau-1}S^m_{1,0}$, from \eqref{eq:derivative tilde p2} we get that $\partial_{x}^\beta\widetilde{q}\in C^{\tau-1-|\delta|} S^m_{1,0}(\R^n\times \R^n;X)=C^{\tau-|\beta|} S^m_{1,0}(\R^n\times \R^n;X)$. Therefore $\widetilde{q}$ is in $C^\tau$ with respect to $x$.\\

\noindent By induction hypothesis
 we can prove from \eqref{eq:derivative tilde p2} that for any $\alpha,\beta\in\N^{n}$, $|\beta|\leq l+1$, there exist a constant $C_{\alpha,\beta}$ and $i'\in\N$ such that 
\begin{equation*}
 \left\|\partial_{\xi}^\alpha\partial_{x}^\beta\widetilde{q}(\cdot,\xi)\right\|_{L^\infty(\R^n;X)}\langle\xi\rangle^{-m+|\alpha|}\leq C_{\alpha,\beta}\left\|q\right\|^{i'}_{C^\tau(\R^n;S^m_{1,0})},
\end{equation*}
and that for any $\beta\in\N^{n}$ with $|\beta|=[\tau]=l+1$
\begin{align}
 &\underset{x^0\not=x^1}{\sup}\dfrac{\left\|\partial_{\xi}^\alpha\partial_x^\beta\widetilde{q}(x^0,\xi)-\partial_{\xi}^\alpha\partial_x^\beta\widetilde{q}(x^1,\xi)\right\|_X}{|x^0-x^1|^{\tau-1}}\langle\xi\rangle^{-m+|\alpha|} \notag \\
 &\leq 
 \underset{x^0\not=x^1}{\sup}\dfrac{\left\|\partial_{\xi}^\alpha\partial_x^\delta\left(\left.\partial_{z_j}{q_{z}}(x^0,\xi)\right|_{z=x^0}\right) -\partial_{\xi}^\alpha\partial_x^\delta\left(\left.\partial_{z_j}{q_{z}}(x^1,\xi)\right|_{z=x^1}\right) 
 \right\|_X}{|x^0-x^1|^{\tau-1}}\langle\xi\rangle^{-m+|\alpha|} \notag \\
 &\quad + \underset{x^0\not=x^1}{\sup}\dfrac{\left\|\partial_{\xi}^\alpha\partial_x^\delta\left(\left.\partial_{x_j}{q_{z}}(x^0,\xi)\right|_{z=x^0}\right) -\partial_{\xi}^\alpha\partial_x^\delta\left(\left.\partial_{x_j}{q_{z}}(x^1,\xi)\right|_{z=x^1}\right) 
 \right\|_X}{|x^0-x^1|^{\tau-1}}\langle\xi\rangle^{-m+|\alpha|} \notag \\
 &\leq \left\| \left.\partial_{\xi}^\alpha\partial_{z_j}q_z(\cdot,\xi)\right|_{z=\cdot}\right\|_{C^{\tau-1}(\R^n;X)} \langle\xi\rangle^{-m+|\alpha|} 
 +\left\|\left.\partial_{\xi}^\alpha\partial_{x_j}q_z(\cdot,\xi)\right|_{z=\cdot}\right\|_{C^{\tau-1}(\R^n;X)}\langle\xi\rangle^{-m+|\alpha|} \notag \\
 &\leq C\left\|q\right\|^{i'}_{C^\tau(\R^n;S^m_{1,0})}. \notag
\end{align}
since for all $|\beta|\leq l+1$, $\partial_z^\beta q\in C^{\tau-|\beta|}_z (\R^n;S^m_{1,0})$.
\end{proof}
The previous argument also shows that for all $i\in\N$, there exist $C\in\R$, $i'\in\N$ such that
\begin{equation*}
 \left\| \widetilde{q}\right\|_{C^{\tau} S^m_{1,0}}^{i} \leq C \left\| q\right\|_{C^{\tau}(\R^n;S^m_{1,0})}^{i'}.
\end{equation*}
Thus we get that the map $\Phi_4$ is well--defined and continuous.
\end{proof}
Altogether we have finished the proof of Theorem~\ref{theorem1}.
\end{proof}

\section{Localizations of Nonsmooth Pseudodifferential Operators}
\label{section:Truncation of nonsmooth pdos}

For the following let $X:= \mathcal{L}(X_0,X_1)$ for two Banach spaces $X_0,X_1$ and $p\in C^\tau S^m_{1,0}\equiv C^\tau S^m_{1,0}(\R^n\times\R^n;X)$ for some $m\in\R$ and $\tau>0$.

The definition of pseudodifferential operators on manifolds involves operators with Schwartz kernels vanishing on the diagonal that appear as remainders, when a pseudodifferential operator is localized. These remainders are compositions of a pseudodifferential operator with appropriate multiplication operators by smooth functions having disjoint supports. In this section we show that any of such remainders of localizations of nonsmooth pseudodifferential operators has a representation by a so-called $C^\tau - C^\infty$--kernel:
\begin{defin}
\label{def:Ctau Cinfty kernel PDOs}
Let $p(x,D_x)$ be a pseudodifferential operator with a symbol  $p \in C^\tau S^m_{1,0}(\Rn\times \Rn;X)$ for some $m\in\R$, $\tau>0$. A function $K:\R^n\times\R^n\to X$ such that for all $f\in\mathcal{S}(\R^n;X_0)$,
 \begin{equation}
\label{eq:Ctau Cinfty kernel p(x,Dx)}
  p(x,D_x)f(x)=\int_{{\R}^n}K(x,x-y)\,f(y)\,dy\qquad \text{for all }x\in\R^n
 \end{equation}
is called a $C^\tau - C^\infty$--kernel of $p(x,D_x)$ if it is in $C^\tau$ with respect to the first variable and it is in $C^\infty$ with respect to the second variable such that for all $\alpha,\beta\in\N^n$, there is some constant $C_{\alpha,\beta}>0$ such that
 \[
  \sup_{z\in{{\R}^n}}\left\|z^\beta\partial_z^\alpha K(\cdot,z)\right\|_{C^\tau({\R}^n;X)}\leq C_{\alpha,\beta}.
 \]
\end{defin}

\begin{lemma}
\label{lem:pdo order -infty has CtauCinfty kernel}
Let $p \in C^\tau S^m_{1,0}(\Rn\times \Rn;X)$ for some $m\in\R$, $\tau>0$. Then $p(x,D_x)$ is of order $-\infty$ (i.e., $p\in C^\tau S^{-\infty}_{1,0}$) if and only if $p(x,D_x)$ has a $C^\tau - C^\infty$--kernel.
\end{lemma}
\begin{proof}
 The proof is similar to \cite[Proof of Theorem 7.6]{Wloka}, one mainly has to take the nonsmoothness with respect to $x$ into account, where $x$ plays the role of one parameter. Let $p\in C^\tau S^{-\infty}_{1,0}({\R}^n\times {\R}^n;X)$ for some $\tau>0$. Define 
\begin{equation}
\label{eq:CtauCinfty kernel CtauS-infty symbol}
 K_p(x,z):=\int_{{\R}^n}e^{iz\xi}p(x,\xi)\,\dbar\xi.
\end{equation}
For all $\alpha,\beta\in\N^n$ the estimate 
 \[
  \underset{z\in{\R}^n}{\sup}\left\|z^\beta\partial_z^\alpha K_p(\cdot,z)\right\|_{C^\tau\left({\R}^n;X\right)}\leq C_{\alpha,\beta},
 \]
can be derived from
\[
 z^\beta\partial_{x}^\gamma\partial_{z}^\alpha K_p(x,z)=\int_{{\R}^n}e^{iz\xi}(i\partial_{\xi})^\beta\Big((i{\xi})^\alpha\partial_{x}^\gamma p(x,\xi)\Big)\dbar\xi,
\]
which follows from integration by parts and the identity $z^\beta e^{iz\xi}=(-i\partial_{\xi})^\beta e^{iz\xi}$.
Indeed, for all $|\gamma|\leq[\tau]$, setting $D_{\xi}^\beta:=(-i)^{|\beta|}\partial_{\xi}^\beta$,
\begin{align*}
 &\underset{x\in{\R}^n}{\sup}\left\|z^\beta\partial_{x}^\gamma\partial_{z}^\alpha K_p(x,z)\right\|_X\\
&\quad =\underset{x\in{{\R}^n}}{\sup}\left\|\int_{{\R}^n}e^{iz\xi}(i\partial_{\xi})^\beta\Big((i{\xi})^\alpha\partial_{x}^\gamma p(x,\xi)\Big)\,\dbar\xi\right\|_X\leq C\,\underset{x\in{\R}^n}{\sup}\int_{{\R}^n}\left\|D_{\xi}^\beta\Big({\xi}^\alpha D_{x}^\gamma p(x,\xi)\Big)\right\|_X\,\dbar\xi\\
 &\quad\leq C\,\sum_{|\delta|\leq|\beta|}\int_{{\R}^n}\left|D_{\xi}^\delta({\xi}^\alpha)\right|\left\|D_{\xi}^{\beta-\delta}D_{x}^\gamma p(\cdot,\xi)\right\|_{L^\infty({\R}^n;X)}\,\dbar\xi
\leq C\int_{{\R}^n}\langle\xi\rangle^{-N}\,\dbar\xi 
\leq C
\end{align*}
independently of $z$ for any $N\geq n+1$. Therefore, for all $|\gamma|\leq[\tau]$ there exists a constant $C$ such that
\[
 \left\|z^\beta\partial_{x}^\gamma\partial_{z}^\alpha K_p(\cdot,z)\right\|_{L^\infty\left({\R}^n;X\right)}<C.
\]
We also have for $|\gamma|=[\tau]$
\begin{align*}
 &\underset{x^0\not=x^1}{\sup}\dfrac{\left\|z^\beta\partial_{x}^\gamma\partial_{z}^\alpha K_p(x^0,z)-z^\beta\partial_{x}^\gamma\partial_{z}^\alpha K_p(x^1,z)\right\|_X}{|x^0-x^1|^{\tau-[\tau]}} \\
 &=\underset{x^0\not=x^1}{\sup}\dfrac{\left\|\int_{{\R}^n}e^{iz\xi}\left[(i\partial_{\xi})^\beta\Big((i{\xi})^\alpha\partial_{x}^\gamma p(x^0,\xi)\Big)- (i\partial_{\xi})^\beta\Big((i{\xi})^\alpha\partial_{x}^\gamma p(x^1,\xi)\Big)\right]\,\dbar\xi\right\|_X}{|x^0-x^1|^{\tau-[\tau]}}\\
 &\leq C\int_{{\R}^n}\underset{x^0\not=x^1}{\sup}\dfrac{\left\|D_{\xi}^\beta\Big({\xi}^\alpha D_{x}^\gamma p(x^0,\xi)\Big)- D_{\xi}^\beta\Big({\xi}^\alpha D_{x}^\gamma p(x^1,\xi)\Big)\right\|_X}{|x^0-x^1|^{\tau-[\tau]}}\,\dbar\xi\\
 &\leq C\sum_{|\delta|\leq|\beta|}\int_{{\R}^n}\left|D_{\xi}^\delta({\xi}^\alpha)\right|\left\|D_{\xi}^{\beta-\delta} p(\cdot,\xi)\right\|_{C^\tau({\R}^n;X)}\,\dbar\xi
 \leq C\int_{{\R}^n}\langle\xi\rangle^{-N}\,\dbar\xi\leq C
\end{align*}
independently of $z$ for any $N\geq n+1$. Therefore, for all $\alpha,\beta\in\N^n$ there exists a constant $C_{\alpha,\beta}$ such that
\[
 \underset{z\in{\R}^n}{\sup}\left\|z^\beta\partial_{z}^\alpha K_p(\cdot,z)\right\|_{C^\tau\left({\R}^n;X\right)}<C_{\alpha,\beta}.
\]
Thus, $K_p$ is a $C^\tau - C^\infty$--kernel of the operator $p(x,D_x)$.\\
For the converse, we apply the Fourier transform with respect to $z$ on both sides of \eqref{eq:CtauCinfty kernel CtauS-infty symbol}, and proceeding as before we obtain that for all $i\in\N$, for all $N\in\N$
\[
 |p|_{C^\tau S^{-N}_{1,0}}^{i}:=\underset{|\alpha|\leq i}{\max}\ \underset{\xi\in{\R}^n}{\sup}\left\lVert\partial_{\xi}^\alpha p(\cdot,\xi)\right\lVert_{C^\tau({\R}^n;X)}\langle\xi\rangle^{N+|\alpha|}<\infty.
\]
Therefore $p(x,\xi)\in C^\tau({\R}^n_x;\mathcal{S}({\R}^n_\xi))$, i.e.,\ $p$ is of order $-\infty$.
\end{proof}

\begin{prop}
\label{prop:kernel phi p psi}
 Let $\varphi,\psi\in C_{b}^\infty({\R}^n)$ be such that $\supp\varphi\cap\supp\psi=\emptyset$ and let $p\in C^\tau S^m_{1,0}(\Rn\times \Rn;X)$ for some $\tau>0$, $m\in\R$. Then the operator $\varphi(x)p(x,D_x)\psi(x)$ has a $C^\tau - C^\infty$--kernel $K:\R^n\times\R^n\to X$.
\end{prop}

\begin{proof}
 For some $\delta>0$, $\varphi(x)p(x,\xi)\psi(y)=0$ in $\{(x,y)\in\R^n\times\R^n:|x-y|<\delta\}$, and for any $k\in\N$, $|x-y|^{-2k}$ is smooth in $\{(x,y)\in\R^n\times\R^n:y\not=x\}$ and all derivatives are bounded in $\{(x,y)\in\R^n\times\R^n:|x-y|\geq\delta\}$. Then, for any $f\in\mathcal{S}(\R^n)$, the function $|x-y|^{-2k}\varphi(x)p(x,\xi)\psi(y)f(y)$ belongs to the space of amplitudes $\mathcal{A}_{0,0}^m$ with respect to $(y,\xi)\in\R^n\times\R^n$. Hence, as in \eqref{eq:oscillatory integral} we can write
\begin{align}
 \varphi(x)p(x,D_x)(\psi f)(x) \notag 
 &={\rm Os} - \iint e^{i(x-y)\cdot\xi}\varphi(x)p(x,\xi)\psi(y)f(y)\,dy\,\dbar\xi \notag \\
 &={\rm Os} - \iint |x-y|^{2k}e^{i(x-y)\cdot\xi}|x-y|^{-2k}\varphi(x)p(x,\xi)\psi(y)f(y)\,dy\,\dbar\xi  \notag \\
 &={\rm Os} - \iint e^{i(x-y)\cdot\xi}(-\Delta_\xi)^kp(x,\xi)|x-y|^{-2k}\varphi(x)\psi(y)f(y)\,dy\,\dbar\xi  \label{eq: Os1}
\end{align} 
by using integration by parts. We also have
\[
 (-\Delta_\xi)^kp(x,\xi)=\mathcal{O}(\langle\xi\rangle^{m-2k}).
\]
We can choose $k$ such that $m-2k\leq -n-1$, and $-2k\leq-n-1$. Then
\[
 \left|(-\Delta_\xi)^kp(x,\xi)|x-y|^{-2k}\varphi(x)\psi(y)f(y)\right|\leq C\langle\xi\rangle^{-n-1}|x-y|^{-n-1}\|f\|_{L^\infty(\R^n)},
\]
so the term $(-\Delta_\xi)^kp(x,\xi)|x-y|^{-2k}\varphi(x)\psi(y)f(y)$ belongs to $L^1(\R^n\times\R^n)$ with respect to $(y,\xi)$, and we can use Fubini's theorem to obtain that the right-hand side of \eqref{eq: Os1} coincides with 
\begin{equation}
\label{eq: Os2}
 \iint e^{i(x-y)\cdot\xi}(-\Delta_\xi)^kp(x,\xi)\,\dbar\xi\,|x-y|^{-2k}\varphi(x)\psi(y)f(y)\,dy.
\end{equation}
Considering
\begin{equation}
\label{eq:kernel1}
 \widetilde{K}_k(x,z):=\int_{\R^n} e^{iz\cdot\xi}(-\Delta_\xi)^kp(x,\xi)\,\dbar\xi,
\end{equation}
\eqref{eq: Os2} amounts to
\begin{equation}
\label{eq:kernel tilde K_k}
 \int_{\R^n} \widetilde{K}_k(x,x-y)\,|x-y|^{-2k}\varphi(x)\psi(y)f(y)\,dy. 
\end{equation}

\begin{claim}
\label{claim:tilde K_k is Cinfty}
 The function $\widetilde{K}_k(x,z)$ defined in \eqref{eq:kernel1} is in $C^\infty(\R^n\setminus\{0\};X)$ with respect to $z$ for all $x\in\R^n$, and it is in $C^\tau(\R^n;X)$ with respect to $x$ for all $z\in\R^n$.
\end{claim}

\begin{proof}
First let $\alpha\in\N^n$ be such that $|\alpha|+m-2k<-n-1$ then
 \[
  \partial_z^\alpha \widetilde{K}_k(x,z)=\int_{\R^n}(i\xi)^\alpha e^{iz\cdot\xi}(-\Delta_\xi)^kp(x,\xi)\,\dbar\xi,
 \]
 and therefore
 \begin{align*}
  \left\|\partial_z^\alpha \widetilde{K}_k(\cdot,z)\right\|_{C^\tau(\R^n;X)} 
  &\leq\int_{\R^n}\left\|(i\xi)^\alpha e^{iz\cdot\xi}(-\Delta_\xi)^kp(\cdot,\xi)\right\|_{C^\tau(\R^n)}\,\dbar\xi \\
  &\leq C\int_{\R^n}\langle\xi\rangle^{|\alpha|+m-2k}\,\dbar\xi \leq C\int_{\R^n}\langle\xi\rangle^{-n-1}\,\dbar\xi < \infty.
 \end{align*}

Finally, if $\alpha\in\N^n$ is such that $|\alpha|+m-2k\geq-n-1$, we can choose $\hat{k}\in\N$ such that $|\alpha|+m-2\hat{k}<-n-1$. By  integration by parts
 \begin{align*}
  &|z|^{2\hat{k}-2k}\partial_z^\alpha \widetilde{K}_k(x,z)
  = \int_{\R^n} e^{iz\cdot\xi}(-\Delta_\xi)^{\hat{k}-k}\left[(i\xi)^\alpha(-\Delta_\xi)^kp(x,\xi)\right]\,\dbar\xi.
 \end{align*}
Using the Leibniz rule one observes that the integrand is a sum, where each term is $\mathcal{O}(\langle\xi\rangle^{m-2\hat{k}+|\alpha|})$. Therefore the integrand is in $L^1(\R^n)$ uniformly with respect to $x,z\in\R^n$. Thus,
  \[
    \partial_z^\alpha \widetilde{K}_k(x,z)= |z|^{-2\hat{k}+2k}\int_{\R^n} e^{iz\cdot\xi}(-\Delta_\xi)^{\hat{k}-k}\left[(i\xi)^\alpha(-\Delta_\xi)^kp(x,\xi)\right]\,\dbar\xi,
  \]
where the first term in the product of the right hand side is smooth in $\{z\in\R^n:z\not=0\}$, and in the second term we apply the argument given above for the first case.
\end{proof}

\noindent From \eqref{eq:kernel tilde K_k} we can define the kernel that we need to conclude the proof of the theorem, by
\begin{align}
 K(x,z)&:=\widetilde{K}_k(x,z)\,|z|^{-2k}\varphi(x)\psi(x-z) \notag \\
 &=\int_{\R^n} e^{iz\cdot\xi}(-\Delta_\xi)^kp(x,\xi)\,\dbar\xi\,|z|^{-2k}\varphi(x)\psi(x-z).  \notag
\end{align}
Note that this definition is independent of $k$. Moreover, we have
\begin{claim}
 For all $N\in\N$, the function $K(x,z)$ is in $C^N(\Rn;X)$ with respect to $z$ for all $x\in\R^n$, it is in $C^\tau(\Rn;X)$ with respect to $x$ for all $z\in\R^n$, and for all $\alpha\in\N^n$ such that $|\alpha|\leq N$ and for all $M\in\N$ there exists a constant $C_{\alpha,M}$ such that
\[
 \left\|\partial_z^\alpha K(\cdot,z)\right\|_{C^\tau(\R^n;X)}\leq C_{\alpha,M}\langle z\rangle^{-M}\qquad \text{ for all } z\in\R^n.
\]
\end{claim}

\begin{proof}
Choose $k$ such that $m+N-2k<-n-1$. The statement follows from Claim \ref{claim:tilde K_k is Cinfty} since $K(x,z)$ is the product of $\widetilde{K}_k(x,z)$ with the functions $|z|^{-2k}$, $\varphi(x)$, and $\psi(x-z)$, which are smooth on $\R^n\setminus\{0\}$.
\end{proof}
\noindent Altogether $K$ is a $C^\tau - C^\infty$--kernel of the operator $\varphi(x)p(x,D_x)\psi(x)$.
\end{proof}

\section{Nonsmooth Green Operators}\label{sec:Green}

We briefly recall the definition of nonsmooth Green operators as they were introduced in \cite{NonsmoothGreen}. We refer to  \cite[Section 1.2]{FunctionalCalculus} for the treatment of Green operators in the smooth case. 

Recall that $\mathcal{S}_+:=\mathcal{S}(\overline{\R_+})$\index{$\mathcal{S}_+$} denotes the space of smooth rapidly decreasing functions on $\overline{\R_+}$. 

\begin{defin}[{\cite[Def.\ 4.1]{NonsmoothGreen}}]
\label{def:poisson trace Ctau}
 Let $\tau>0$, $m\in\R$. Let $U$ be $\R^{n-1}$ or $\overline{\R^n_+}$. The set of symbol--kernels $C^\tau S^m_{1,0}(U\times \R^{n-1};\mathcal{S}_{+})$\index{$C^\tau S^m_{1,0}(U\times \R^{n-1};\mathcal{S}_{+})$} is the set of functions $h(x,\xi',y_n)$, which are smooth in $(\xi',y_n)\in \R^{n-1}\times\overline{\R}_{+}$ and which are in $C^\tau(U)$ with respect to $x$ satisfying for all $\alpha\in\N^{n-1}, l,l'\in\N$ the estimate
\begin{equation}
\label{estimation2poissontrace}
 \left\lVert y_n^l\partial_{y_n}^{l'}\partial_{\xi'}^\alpha h(\cdot,\xi',\cdot)\right\lVert_{C^\tau(U;L^2_{y_n}({\R}_{+}))}\leq C_{\alpha,l,l'}\langle\xi'\rangle^{m+\frac{1}{2}-l+l'-|\alpha|}
\end{equation}
for all $\xi'\in\R^{n-1}$, for some constant $C_{\alpha,l,l'}$. 
\end{defin}
We note that a function $h$ belongs to $C^\tau S^{m}_{1,0}(U\times \R^{n-1};\mathcal{S}_{+})$ if and only if for all $l,l'\in\N$
\[
 y_n^l\partial_{y_n}^{l'}{h}\in C^\tau S^{m+\frac{1}{2}-l+l'}_{1,0}(U\times \R^{n-1};L^2_{y_n}({\R}_{+})).
\]
Moreover, we have
\[
 S^m_{1,0}(U\times \R^{n-1};\mathcal{S}_{+})=\bigcap_{\tau>0}C^\tau S^m_{1,0}(U\times \R^{n-1};\mathcal{S}_{+}),
\]
where $S^m_{1,0}(U\times \R^{n-1};\mathcal{S}_{+})$\index{$S^m_{1,0}(U\times \R^{n-1};\mathcal{S}_{+})$} is the smooth symbol--kernel space as e.g.\ in \cite[Section 2.3]{FunctionalCalculus}. We will use the shorthand notation $C^\tau S^m_{1,0}(\mathcal{S}_+)$\index{$C^\tau S^m_{1,0}(\mathcal{S}_+)$} for $C^\tau S^m_{1,0}(U\times \R^{n-1};\mathcal{S}_{+})$,  and $S^m_{1,0}(\mathcal{S}_+)$\index{$S^m_{1,0}(\mathcal{S}_+)$} for $S^m_{1,0}(U\times \R^{n-1};\mathcal{S}_{+})$. 
 We equip the space $C^\tau S^m_{1,0}(\mathcal{S}_+)$ with the following seminorms 
\begin{align*}
\label{eq:seminorms CtauSm1,0 poissontrace}
 |h|_{C^\tau S^m_{1,0}(\mathcal{S}_+)}^{i}:= \underset{l,l',|\alpha|\leq i}{\max}\ \underset{\xi'\in \R^{n-1}}{\sup}\left\|y_n^l\partial_{y_n}^{l'}\partial_{\xi'}^\alpha{h}(\cdot,\xi',\cdot)\right\|_{C^\tau(U;L^2_{y_n}({\R}_{+}))} \langle\xi'\rangle^{-m-\frac{1}{2}+l-l'+|\alpha|}
\end{align*}
for all $h\in C^\tau S^m_{1,0}(\mathcal{S}_+)$ and $i\in\N$, and the space $S^m_{1,0}(\mathcal{S}_+)$ with the following seminorms
\begin{align*}
 & |h|_i^{(m)}:= \underset{l,l',|\alpha|,|\beta|\leq i}{\max}\ \underset{\substack{x\in U \\ \xi'\in \R^{n-1}}}{\sup}\left\|y_n^l\partial_{y_n}^{l'}\partial_{x}^\beta\partial_{\xi'}^\alpha{h}(x,\xi',\cdot)\right\|_{L^2_{y_n}({\R}_{+})} \langle\xi'\rangle^{-m-\frac{1}{2}+l-l'+|\alpha|}
\end{align*}
for all $h\in S^m_{1,0}(\mathcal{S}_+)$ and $i\in\N$.

\subsection{Nonsmooth Poisson Operators}

\begin{defin}
Let $k\in C^\tau S^{m-1}_{1,0}(\overline{\R^n_+}\times\R^{n-1};\mathcal{S}_{+})$, $\tau>0$, $m\in\R$. Then we define the nonsmooth Poisson operator  of order $m$ on $v\in\mathcal{S}(\R^{n-1})$ associated to the symbol--kernel $k$ by
\[
 (k(x,D_{x'})v)(x',x_n):=\int_{\R^{n-1}}e^{ix'\cdot\xi'}{k}(x,\xi',x_n)\mathcal{F}_{x'\mapsto\xi'}[v](\xi')\,\dbar\xi'.
\]
\end{defin}
Using the boundary symbol operator from $\C$ to $C^\tau(\overline{\R_+})$
\[
 k(x,\xi',D_n)a:={k}(x,\xi',x_n)\cdot a\quad \text{ for }a\in\C
\]
for fixed $\xi',x'\in\R^{n-1}$, we can also express the operator $k(x,D_{x'})$ in the form
\[
 k(x,D_{x'})=\Op'(k(x,\xi',D_n)),
\]
where $\Op'(k(x,\xi',D_n))$\index{$\Op'$} denotes the pseudodifferential operator with respect to $x\in\overline{\R^n_+}$, $\xi'\in \R^{n-1}$ associated to $k(x,\xi',D_n)\in C^\tau S^{m-1}_{1,0}(\R^{n-1}\times \R^{n-1};\mathcal{L}(\C, L^2(\R_+)))$.

\subsection{Nonsmooth Trace Operators}

\begin{defin}
Let $t\in C^\tau S^m_{1,0}(\R^{n-1}\times\R^{n-1};\mathcal{S}_{+})$, $\tau>0$, $m\in\R$. Then we define the nonsmooth trace operator of order $m$ and class $0$ on $u\in\mathcal{S}(\overline{\R^n_+})$ associated to the symbol--kernel $t$ by
\[
 t(x',D_{x})(u)(x'):=\int_{\R^{n-1}}e^{ix'\cdot\xi'}\int_0^\infty{t}(x',\xi',y_n)\mathcal{F}_{x'\mapsto\xi'}[u(\cdot,y_n)]\,dy_n\,\dbar\xi'.
\]
\end{defin}
Using the boundary symbol operator
\[
 t(x',\xi',D_n)f:=\int_0^\infty{t}(x',\xi',y_n)f(y_n)\,dy_n\quad \text{ for }f\in\mathcal{S}_+,
\]
for fixed $\xi',x'\in\R^{n-1}$, we can also express the operator $t(x',D_{x})$ in the form
\[
 t(x',D_{x})=\Op'(t(x',\xi',D_n)),
\]
where $\Op'(t(x',\xi',D_n))$ denotes the pseudodifferential operator with respect to $x',\xi'\in \R^{n-1}$ associated to $t(x',\xi',D_n)\in C^\tau S^m_{1,0}(\R^{n-1}\times \R^{n-1};\mathcal{L}(L^2(\R_+),\C))$.

More generally a nonsmooth trace operator of order $m$ and class $r\in\N$ is of the form
\begin{equation}
\label{eq:higher class trace operators}
  t(x',D_{x}) u= \sum_{j=0}^{r-1} s_j(x',D_{x'})\partial_{x_n}^j u|_{x_n=0}+ t_0(x',D_{x})u\quad\text{for all }u\in \mathcal{S}(\overline{\R^n_+}),
\end{equation}
where $s_j\in C^\tau S^{m-j}_{1,0}(\R^{n-1}\times \R^{n-1})$ for all $j=0,\ldots, r-1$ and $t_0(x',D_{x})$ is a nonsmooth trace operator of order $m$ and class $0$.

\subsection{Nonsmooth Singular Green Operators}

We use the notation ${\R}_{++}^2:={\R}_+\times{\R}_+$, $\overline{\R_{++}^2}:=\overline{\R_+}\times\overline{\R_+}$. Since $\mathcal{S}(\overline{\R_+})$ is a nuclear space, $\mathcal{S}(\overline{\R_+})\hat{\otimes}\,\mathcal{S}(\overline{\R_+})=\mathcal{S}(\overline{\R_{++}^2})=:\mathcal{S}_{++}$\index{\mathcal{S}_{++}}. 

\begin{defin}[{\cite[Def.\ 4.1]{NonsmoothGreen}}]
\label{def:sgo Ctau}
Let $\tau>0$, $m\in\R$. The set $C^\tau S^m_{1,0}(\overline{\R^n_+}\times\R^{n-1};\mathcal{S}_{++})$, $\tau>0$, $m\in\R$, is the set of functions $g(x,\xi',y_n,w_n)$, which are smooth in $(\xi',y_n,w_n)\in \R^{n-1}\times\overline{\R_{++}^2}$ and which are in $C^\tau(\overline{\R^n_+})$ with respect to $x$ satisfying for all $\alpha\in\N^{n-1}, k,k',l,l'\in\N$ the estimate
\begin{align}
 \left\lVert y_n^k\partial_{y_n}^{k'}w_n^l\partial_{w_n}^{l'}\partial_{\xi'}^\alpha g(\cdot,\xi',\cdot,\cdot)\right\lVert_{C^\tau(\overline{\R^n_+};L^2_{y_n,w_n}({\R}_{++}^2))}  \leq \ C_{\alpha,k,k',l,l'}\langle\xi'\rangle^{m+1-k+k'-l+l'-|\alpha|}. \label{estimation2sgo}
\end{align}
for all $\xi'\in\R^{n-1}$, for some constant $C_{\alpha,k,k',l,l'}$. 
\end{defin}
We note that a function $g$ belongs to $C^\tau S^{m}_{1,0}(\overline{\R^n_+}\times\R^{n-1};\mathcal{S}_{++})$ if and only if for all $k,k',l,l'\in\N$
\[
 y_n^k\partial_{y_n}^{k'}w_n^l\partial_{w_n}^{l'}{g}\in C^\tau S^{m+1-k+k'-l+l'}_{1,0}(\overline{\R^n_+}\times\R^{n-1};L^2_{y_n,w_n}({\R}_{++}^2)).
\]
Moreover, we have
\[
 S^m_{1,0}(\overline{\R^n_+}\times\R^{n-1};\mathcal{S}_{++})=\bigcap_{\tau>0}C^\tau S^m_{1,0}(\overline{\R^n_+}\times\R^{n-1};\mathcal{S}_{++}),
\]
where $S^m_{1,0}(\overline{\R^n_+}\times\R^{n-1};\mathcal{S}_{++})$ is the corresponding smooth symbol--kernel class, cf.\ e.g.\ \cite[Section 2.3]{FunctionalCalculus}.
We will use the shorthand notation $C^\tau S^m_{1,0}(\mathcal{S}_{++})$\index{$C^\tau S^m_{1,0}(\mathcal{S}_{++})$} for $C^\tau S^m_{1,0}(\overline{\R^n_+}\times\R^{n-1};\mathcal{S}_{++})$, and $S^m_{1,0}(\mathcal{S}_{++})$\index{$S^m_{1,0}(\mathcal{S}_{++})$} for $S^m_{1,0}(\overline{\R^n_+}\times\R^{n-1};\mathcal{S}_{++})$. We equip the space $C^\tau S^m_{1,0}(\mathcal{S}_{++})$ with the following seminorms
\begin{align*}
 |g|_{C^\tau S^m_{1,0}(\mathcal{S}_{++})}^{i}:= &\underset{k,k',l,l',|\alpha|\leq i}{\max}\ \underset{\xi'\in\R^{n-1}}{\sup}\left\|y_n^k\partial_{y_n}^{k'}w_n^l\partial_{w_n}^{l'}\partial_{\xi'}^\alpha{g}(\cdot,\xi',\cdot,\cdot)\right\|_{C^\tau(\overline{\R^n_+};L^2_{y_n,w_n}({\R}_{++}^2))} \notag \\
 &\phantom{ \underset{k,k',l,l',|\alpha|\leq i}{\max}\ \underset{\xi'\in\R^{n-1}}{\sup} }\cdot\langle\xi'\rangle^{-m-1+k-k'+l-l'+|\alpha|}.
\end{align*}
for all $g\in C^\tau S^m_{1,0}(\mathcal{S}_{++})$ and $i\in\N$, and the space $S^m_{1,0}(\mathcal{S}_{++})$ with the following seminorms
\begin{align*}
 & |g|_i^{(m)}:= \underset{k,k',l,l',|\alpha|,|\beta|\leq i}{\max}\ \underset{\substack{x\in\overline{\R^n_+} \\ \xi'\in \R^{n-1}}}{\sup}\left\|y_n^k\partial_{y_n}^{k'}w_n^l\partial_{w_n}^{l'}\partial_{x}^\beta\partial_{\xi'}^\alpha{g}(x',\xi',\cdot,\cdot)\right\|_{L^2_{y_n,w_n}({\R}_{++}^2)} 
\langle\xi'\rangle^{-m-1+k-k'+l-l'+|\alpha|}
\end{align*}
for all $g\in S^{m}_{1,0}(\mathcal{S}_{++})$ and $i\in\N$. 

\begin{defin}
Let $g\in C^\tau S^{m}_{1,0}(\overline{\R^n_+}\times\R^{n-1};\mathcal{S}_{++})$, $\tau>0$, $m\in\R$. Then we define the nonsmooth singular Green operator of order $m$ and class $0$ on $u\in\mathcal{S}(\overline{\R^n_+})$ associated to the symbol--kernel $g$ by
\[
 g(x,D_{x})(u)(x):=\int_{\R^{n-1}}e^{ix'\cdot\xi'}\int_0^\infty{g}(x,\xi',x_n,w_n)\mathcal{F}_{x'\mapsto\xi'}[u(\cdot,w_n)]\,dw_n\,\dbar\xi'.
\]
\end{defin}
Using the boundary symbol operator
\[
 g(x,\xi',D_n)f:=\int_0^\infty{g}(x,\xi',x_n,w_n)f(w_n)\,dw_n\quad \text{ for }f\in\mathcal{S}_+,
\]
for fixed $x\in\overline{\R^n_+}$, $\xi'\in\R^{n-1}$, we can also express the operator $g(x,D_{x})$ in the form
\[
 g(x,D_{x})=\Op'(g(x,\xi',D_n)),
\]
where $\Op'(g(x,\xi',D_n))$ denotes the pseudodifferential operator with respect to $x\in\overline{\R^n_+}$, $\xi'\in \R^{n-1}$ associated to $g(x,\xi',D_n)$.\\

More generally a nonsmooth singular Green operator of order $m$ and class $r\in\N$ is of the form
\begin{equation*}
  g(x,D_x) u= \sum_{j=0}^{r-1} k_j(x,D_{x'})\partial_{x_n}^j u|_{x_n=0}+ g_0(x,D_x)u\quad\text{for all }u\in \mathcal{S}(\overline{\R^n_+}),
\end{equation*}
where $k_j(x,D_{x'})$ is a Poisson operator of order $m-j$ for all $j=0,\ldots, r-1$ and $g_0(x,D_x)$ is a nonsmooth singular Green operator of order $m$ and class $0$.

\section{Coordinate Changes for Nonsmooth Green Operators}
\label{section:change of variables Green}

In order to show invariance of Green operators with nonsmooth coefficients with respect to suitable coordinate transformations, we proceed as in Section \ref{section:change of variables} with some adaptations.\\

Let $\kappa:\overline{\R^n_+}\to\overline{\R^n_+}$ be a bounded smooth diffeomorphism (see Definition \ref{def:bounded smooth diffeo}). Let us also assume that $\kappa$ extends to a bounded smooth diffeomorphism from $\R^n$ to $\R^n$ and that it preserves the boundary $\partial\overline{\R^n_+}:=\{x\in\overline{\R^n_+}:x_n=0\}$\index{$\partial\overline{\R^n_+}$}, i.e., $\kappa (\R^{n-1}\times \{0\})= \R^{n-1}\times \{0\}$. Along with the diffeomorphism $\kappa$ we consider as in \cite[Section 2.4]{FunctionalCalculus}, the induced diffeomorphism on $\partial\overline{\R^n_+}\cong\R^{n-1}$, which will be called $\lambda$
\begin{equation}
\label{eq:lambda}
 \lambda(x'):=(\kappa_1(x',0),\ldots,\kappa_{n-1}(x',0))\qquad \text{ for all } x'\in\R^{n-1}.
\end{equation}

\begin{thm}
\label{thm: change variables Green}
 Let $\tau>0$, $\tau \not\in\N$ and $m\in\R$.
 \begin{enumerate}
  \item Let $h(x,D_{x'})$ be a nonsmooth Poisson operator of order $m$ with symbol--kernel $h\in C^\tau S^{m-1}_{1,0}(\overline{\R^n_+}\times \R^{n-1};\mathcal{S}_+)$. Then there is some $\widetilde{h}\in C^\tau S^{m-1}_{1,0}(\overline{\R^n_+}\times \R^{n-1};\mathcal{S}_+)$ such that
  \begin{equation}
   \label{eq:change of variables poisson}
   \widetilde{h}(x,D_{x'})v(x):=\kappa^{-1,*}h(x,D_{x'})\lambda^*v(x) \qquad \text{for }v\in\mathcal{S}(\R^{n-1}).
  \end{equation}
  \item Let $h(x',D_{x})$ be a nonsmooth trace operator of order $m$ and class $0$ with symbol--kernel $h\in C^\tau S^{m}_{1,0}(\R^{n-1}\times \R^{n-1};\mathcal{S}_+)$. Then there is some $\widetilde{h}\in C^\tau S^{m}_{1,0}(\R^{n-1}\times \R^{n-1};\mathcal{S}_+)$ such that
   \begin{equation}
    \label{eq:change of variables trace}
    \widetilde{h}(x',D_{x})u(x):=\lambda^{-1,*}h(x',D_{x})\kappa^*u(x') \qquad \text{for }u\in\mathcal{S}(\overline{\R^{n}_+}).
   \end{equation}
  \item Let $h(x,D_{x})$ be a nonsmooth singular Green operator of order $m$ and class $0$ with symbol--kernel $h\in C^\tau S^{m-1}_{1,0}(\overline{\R^n_+}\times \R^{n-1};\mathcal{S}_{++})$. Then there is some $\widetilde{h}\in C^\tau S^{m-1}_{1,0}(\overline{\R^n_+}\times \R^{n-1};\mathcal{S}_{++})$ such that
   \begin{equation}
    \label{eq:change of variables sgo}
    \widetilde{h}(x,D_{x})u(x):=\kappa^{-1,*}h(x,D_{x})\kappa^*u(x) \qquad \text{for }u\in\mathcal{S}(\overline{\R_+^n}).
   \end{equation}
 \end{enumerate}
\end{thm}
\begin{note} In the case $\tau\in \N$ the statements of the theorem hold true if one replaces $C^\tau$ by $C^{\tau-1,1}$, cf.~Remark~\ref{rem:tauN}. The following proofs can be easily carried over to that case.\\ It is also important to remark that the case of operators of general class can be reduced to the case of operators of class 0.
\end{note}

\begin{proof}
The proof of Theorem \ref{thm: change variables Green} follows the same scheme as the proof of Theorem \ref{theorem1}. In the following let $h$ be the symbol--kernel of a nonsmooth Poisson, trace or singular Green operator as above. We use the notation $U:=\overline{\R_+^n}$ and $S^m_{1,0}:=S^m_{1,0}(\mathcal{S}_+)$ if $h$ is the symbol--kernel of a Poisson operator, $U:=\R^{n-1}$  and $S^m_{1,0}:=S^m_{1,0}(\mathcal{S}_{+})$ if $h$ is the symbol--kernel of a trace operator, and $U:=\overline{\R_+^n}$ and $S^m_{1,0}:=S^m_{1,0}(\mathcal{S}_{++})$ if $h$ is the symbol--kernel of a singular Green operator, respectively.

We introduce the following maps:
 \begin{align*}
	\Phi_1\colon C^\tau S^m_{1,0}&\to C^\tau(U;S^m_{1,0})\colon h\mapsto q,
	\end{align*}
	such that for all $z,x\in U$, $\xi'\in \R^{n-1}$, 
	\begin{enumerate}
	 \item $q_z(x,\xi',y_n):=h(z,\xi',y_n)$ for all $y_n\in\R_+$ if $h$ is the symbol--kernel of a Poisson or a trace operator.
	 \item $q_z(x,\xi',y_n,w_n):=h(z,\xi',y_n,w_n)$ for all $y_n,w_n\in\R_+$  if $h$ is the symbol--kernel of a singular Green operator.
	\end{enumerate}
 \begin{align*}
	\Phi_2\colon C^\tau(U;S^m_{1,0})&\to C^\tau(U;S^m_{1,0})\colon q\mapsto \Phi_2(q),
	\end{align*}
	where for all $z,x\in U$, $\xi'\in \R^{n-1}$, 
	\begin{enumerate}
	 \item $(\Phi_2(q))_z(x,\xi',y_n):= q_{\kappa^{-1}(z)}(x,\xi',y_n)$ for all $y_n\in\R_+$ for Poisson operators.
	 \item $(\Phi_2(q))_z(x,\xi',y_n):= q_{\lambda^{-1}(z)}(x,\xi',y_n)$ for all $y_n\in\R_+$ for trace operators.
	 \item $(\Phi_2(q))_z(x,\xi',y_n,w_n):= q_{\kappa^{-1}(z)}(x,\xi',y_n,w_n)$ for all $y_n,w_n\in\R_+$ for singular Green operators.
	\end{enumerate}
  \begin{align*}
	\Phi_3\colon C^\tau(U;S^m_{1,0})&\to C^\tau(U;S^m_{1,0})\colon q\mapsto \Phi_3(q),
	\end{align*}
	where for all $z,x\in U$, $\xi'\in \R^{n-1}$, 
	\begin{enumerate}
	 \item $(\Phi_3(q))_z(x,\xi',y_n):= (Tq_z)(x,\xi',y_n)$ for all $y_n\in\R_+$, with
	  \[
	    (Th)(x,D_{x'}):=\kappa^{-1,*}h(x,D_{x'})\lambda^*\ \ \text{ for all } h\in S^m_{1,0}(\overline{\R^n_+}\times \R^{n-1};\mathcal{S}_+)
	  \]
	  for Poisson operators.
	 \item $(\Phi_3(q))_z(x,\xi',y_n):= (Tq_z)(x,\xi',y_n)$ for all $y_n\in\R_+$, with
	  \[
	    (Th)(x',D_{x}):=\lambda^{-1,*}h(x',D_{x})\kappa^*\ \ \text{ for all } h\in S^m_{1,0}(\R^{n-1}\times \R^{n-1};\mathcal{S}_+)
	  \]
	  for trace operators.
	 \item $(\Phi_3(q))_z(x,\xi',y_n,w_n):= (Tq_z)(x,\xi',y_n,w_n)$ for all $y_n, w_n\in\R_+$, with
	  \[
	    (Th)(x,D_{x}):=\kappa^{-1,*}h(x,D_{x})\kappa^*\ \ \text{ for all } h\in S^m_{1,0}(\overline{\R^n_+}\times \R^{n-1};\mathcal{S}_{++})
	  \]
	  for singular Green operators.
	\end{enumerate}
  \begin{align*}
        \Phi_4\colon C^\tau(U;S^m_{1,0})&\to C^\tau S^m_{1,0}\colon q\mapsto \Phi_4(q),
       \end{align*}
	where for all $z,x\in U$, $\xi'\in \R^{n-1}$, 
	\begin{enumerate}
	 \item $(\Phi_4(q))(x,\xi',y_n):=\left.q_z(x,\xi',y_n)\right|_{z=x}$ for all $y_n\in\R_+$ for Poisson and trace operators.
	 \item $(\Phi_4(q))(x,\xi',y_n,w_n):=\left.q_z(x,\xi',y_n,w_n)\right|_{z=x}$ for all $y_n, w_n\in\R_+$ for singular Green operators.
	\end{enumerate}
In the following we adapt the proofs of Section \ref{section:change of variables} and show that these maps are well--defined and continuous. Then the symbol--kernels of the operators given by \eqref{eq:change of variables poisson}, \eqref{eq:change of variables trace} and \eqref{eq:change of variables sgo}, can be written as
\[
 \widetilde{h}=\Phi_4(\Phi_3(\Phi_2(\Phi_1(h)))),
\]
and we conclude the statement of the theorem.

As before $\Phi_1$ corresponds to a ``freezing of coefficients'' by looking at $q_z=h(z,\cdot)$ as a smooth, $x$-independent symbol--kernel, parametrized by the spatial variable $z$. Moreover,  $\Phi_2$ and $\Phi_3$ treat the coordinate transformations with respect to the  spatial variable $z\in\Rn$ and the smooth symbol--kernel $q_z$, respectively. Finally, $\Phi_4$ corresponds to ``unfreezing'' the coefficients.

The proofs of the following lemmas are an adaptation of the proofs of the corresponding lemmas from Section \ref{section:change of variables} introducing the new variables resp.\ $y_n$ for Poisson or trace operators, resp.\ $(y_n,w_n)$ for singular Green operators as well as the factor resp.\ $y_n^k\partial_{y_n}^{k'}$ for Poisson or trace operators, resp.\ $y_n^k\partial_{y_n}^{k'}w_n^l\partial_{w_n}^{l'}$ for singular Green operators, in the symbol--kernel estimates and an additional term resp.\ $-k+k'$ for Poisson or trace operators, resp.\ $-k+k'-l+l'$ for singular Green operators, in the exponent of $\langle\xi'\rangle$, in a straightforward manner. For the convenience of the reader we give the details for the case of singular Green operators. In the following we set $X_1:=\mathcal{S}_{++}$ and $X_2:=L^2_{y_n,w_n}({\R}_{++}^2)$.

Again the continuity of the mappings $\Phi_2$ and $\Phi_3$ can be easily verified since they only act with respect to $z\in\overline{\R^n_+}$, $q_z\in S^m_{1,0}$, respectively. 

\begin{lemma}
\label{lemma:Phi1 sgo}
 Let $m\in\R$, $\tau>0$, $\tau\notin\N$. The map $\Phi_1\colon C^\tau S^m_{1,0}\to C^\tau(\overline{\R^n_+};S^m_{1,0})$ is well--defined and continuous.
\end{lemma}

\begin{proof}
 For $z\in\overline{\R^n_+}$ fixed, the function $q_{z}$ lies in $S^m_{1,0}$. In fact, $q_{z}$ has constant coefficients and therefore it is smooth with respect to $x$. Since $g\in C^\tau S^m_{1,0}$, for any $\alpha\in\N^{n-1}$, $\beta\in\N^{n}$, $k,k',l,l'\in\N$ there exists a constant $C$ such that
\begin{align*}
 \left\|y_n^k\partial_{y_n}^{k'}w_n^l\partial_{w_n}^{l'}\partial_{\xi'}^\alpha\partial_{x}^\beta q_{z}(x,\xi',\cdot,\cdot)\right\|_{X_2} 
 &=\left\|y_n^k\partial_{y_n}^{k'}w_n^l\partial_{w_n}^{l'}\partial_{\xi'}^\alpha\partial_{x}^\beta \left(h(z,\xi',\cdot,\cdot)\right)\right\|_{X_2} \\ 
 &\leq C\langle\xi'\rangle^{m+1-k+k'-l+l'-|\alpha|}
\end{align*}
 which implies that $q_{z}\in S^m_{1,0}$ for all $z\in\overline{\R^n_+}$.\\

 By definition of $q_z$, if $h\in C^\tau$ with respect to $x$, then $q\in C^\tau$ with respect to $z$. Indeed, for all $i\in\N$ and for all $\delta\in\N^{n}$ such that $|\delta|\leq[\tau]$
\begin{align*}
 &\underset{z\in\overline{\R^n_+}}{\sup}|\partial_{z}^\delta q_z|^{(m)}_i\\
 &=\underset{z\in\overline{\R^n_+}}{\sup}\,\underset{k,k',l,l',|\alpha|,|\beta|\leq i}{\max}\,\underset{\substack{x\in\overline{\R^n_+}  \xi'\in \R^{n-1}}}{\sup}\left\|y_n^k\partial_{y_n}^{k'}w_n^l\partial_{w_n}^{l'}\partial_{\xi'}^\alpha\partial_{x}^\beta\partial_{z}^\delta q_z(x,\xi',\cdot,\cdot)\right\|_{X_2}
\langle\xi'\rangle^{-m-1+k-k'+l-l'+|\alpha|} \\
 &=\underset{z\in\overline{\R^n_+}}{\sup}\,\underset{k,k',l,l',|\alpha|\leq i}{\max}\,\underset{\xi'\in \R^{n-1}}{\sup}\left\|y_n^k\partial_{y_n}^{k'}w_n^l\partial_{w_n}^{l'}\partial_{\xi'}^\alpha\partial_{z}^\delta \left(h(z,\xi',\cdot,\cdot)\right)\right\|_{X_2}
\langle\xi'\rangle^{-m-1+k-k'+l-l'+|\alpha|} \\
 &\leq C\,|h|_{C^\tau S^m_{1,0}}^{i},
\end{align*}
Moreover for $z^0,z^1\in\overline{\R^n_+}$ and for all $\delta\in\N^{n}$ such that $|\delta|=[\tau]$,
\begin{align*}
 &|\partial_{z}^\delta q_{z^0}-\partial_{z}^\delta q_{z^1}|_i^{(m)} \\
 &=\underset{k,k',l,l',|\alpha|,|\beta|\leq i}{\max}\underset{\substack{x\in\overline{\R^n_+} \\ \xi'\in \R^{n-1}}}{\sup}\left\|y_n^k\partial_{y_n}^{k'}w_n^l\partial_{w_n}^{l'}\partial_{\xi'}^\alpha\partial_{x}^\beta\partial_{z}^\delta\Big({q_{z^0}}(x,\xi',\cdot,\cdot)-{q_{z^1}}(x,\xi',\cdot,\cdot)\Big)\right\|_{X_2} \notag \\
 &\qquad \cdot
\langle\xi'\rangle^{-m-1+k-k'+l-l'+|\alpha|} \notag \\
 &=\underset{k,k',l,l',|\alpha|,|\beta|\leq i}{\max}\underset{\xi'\in \R^{n-1}}{\sup}\left\|y_n^k\partial_{y_n}^{k'}w_n^l\partial_{w_n}^{l'}\partial_{\xi'}^\alpha\partial_{z}^\delta\Big( {h}(z^0,\xi',\cdot,\cdot)- h(z^1,\xi',\cdot,\cdot)\Big)\right\|_{X_2} \notag \\
 &\qquad \cdot\langle\xi'\rangle^{-m-1+k-k'+l-l'+|\alpha|} \notag \\
 &\leq\ C\ |z^0-z^1|^{\tau-[\tau]}|h|_{C^\tau S^m_{1,0}}^{i}.
\end{align*}
Hence, for all $i\in\N$ and for all $\delta\in\N^n$ such that $|\delta|=[\tau]$
\[
 \underset{\substack{z^0,z^1\in\overline{\R^n_+} \\ z^0\not=z^1}}{\sup}\dfrac{|\partial_{z}^\delta q_{z^0}-\partial_{z}^\delta q_{z^1}|^{(m)}_i}{|z^0-z^1|^{\tau-[\tau]}}\leq\ C\,|h|_{C^\tau S^m_{1,0}}^{i}.
\]
\end{proof}

\begin{lemma}
\label{lemma:Psi sgo}
 The map
 \begin{align}
  \Phi_4:C^\tau(\overline{\R^n_+};S^m_{1,0})&\to C^\tau S^m_{1,0} \notag \\
 (\Phi_4(q))(x',\xi',y_n,w_n)&:=\widetilde{q}(x',\xi',y_n,w_n):=\left.q_z(x',\xi',y_n,w_n)\right|_{z=x}. \notag 
 \end{align}
 is well--defined and continuous.
\end{lemma}

\begin{proof} This follows from the following three claims:
\begin{claim}
\label{estimate norm infty tilde p sgo}
 Let $q\in C^\tau(\overline{\R^n_+};S^m_{1,0})$. For all $i\in\N$ there exists a constant $C\in\R$ such that
\[
 \underset{|\alpha|\leq i}{\max}\ \underset{\xi\in \R^{n-1}}{\sup}\ \left\|\partial_\xi^\alpha\widetilde{q}(\cdot,\xi,\cdot,\cdot)\right\|_{L^\infty(\overline{\R^n_+};X_2)}\langle\xi\rangle^{-m+k-k'+l-l'+|\alpha|} \leq C\left\|q_\bullet\right\|^i_{L^\infty(\overline{\R^n_+};S^m_{1,0})}.
\]
\end{claim}

\begin{claim}
\label{Case tau in (0,1) sgo}
 Let $q\in C^\tau(\overline{\R^n_+};S^m_{1,0})$ for $\tau\in(0,1)$. Then $\widetilde{q}\in C^\tau S^m_{1,0}$.
\end{claim}

\begin{claim}
\label{Case tau not in N sgo}
 Let $q\in C^\tau(\overline{\R^n_+};S^m_{1,0})$ for $\tau>0$, $\tau\notin\N$. Then $\widetilde{q}\in C^\tau S^m_{1,0}$.
\end{claim}

\noindent The proofs of Claim \ref{estimate norm infty tilde p sgo}, Claim \ref{Case tau in (0,1) sgo} and Claim \ref{Case tau not in N sgo} are done in the same way as the proofs of \eqref{estimate norm infty tilde p}, Claim \ref{Case tau in (0,1)} and Claim \ref{Case tau not in N} respectively, introducing the new variables $(y_n,w_n)$ as well as $y_n^k\partial_{y_n}^{k'}w_n^l\partial_{w_n}^{l'}$ in the symbol--kernel estimates and an additional term $-k+k'-l+l'$ in the exponent of $\langle\xi'\rangle$. This also concludes the proof of the lemma.
\end{proof}
Altogether Theorem \ref{thm: change variables Green} is also proven.
\end{proof}

\section{Localizations of  Nonsmooth Green Operators}\label{sec:Localization}

In the introduction of Section \ref{section:Truncation of nonsmooth pdos}, we mentioned that by \emph{remainder of a localization} of an operator we mean the composition of an operator with appropriate multiplication operators, i.e.\  operators representing multiplication by smooth functions with disjoint supports. We also consider this definition here for Green operators. In this section we consider \emph{the truncation of an operator}, which is the restriction of the operator to a subset of the whole space. When we have an operator $P$ defined on $\R^n$, its truncation $P_+$ to $\overline{\R_+^n}$ is the composition $P_+:=r_+\circ P\circ e_+$, where $e_+$ denotes the extension by zero from  $\overline{\R_+^n}$ to $\R^n$, and $r_+$ denotes the restriction of $\R^n$ to ${\R_+^n}:=\{x=(x_1,\ldots,x_n)\in\R^n:x_n>0\}$\index{${\R_+^n}$}. We start with the definition of the kernel representation of some Green operators with $C^\tau$--coefficients. \\

\begin{defin}
\label{def:Green op CtauCinfty kernel}
 A Green operator $a(x,D_x)=\twobytwo{p(x,D_x)_++g(x,D_x)}{k(x,D_{x'})}{t(x',D_x)}{s(x',D_{x'})}$ with $C^\tau$--coefficients is said to have a $C^\tau - C^\infty$--kernel if it satisfies the following: 
 \begin{enumerate}
 \item  $p(x,D_x)$ is a nonsmooth pseudodifferential operator on ${\R}^n$ with $C^\tau - C^\infty$--kernel $K_p$ as in Definition \ref{def:Ctau Cinfty kernel PDOs}. Its truncation $p(x,D_x)_+$\index{$p(x,D_x)_+$} to $\overline{\R_+^n}$ is given by
 \[
  p(x,D_x)_+f(x):=\int_{\overline{\R_+^n}}K_p(x,x-y)\,f(y)\,dy,
 \]
 for all $x\in{\overline{\R_+^n}}$ and for any $f\in\mathcal{S}(\overline{\R_+^n})$;
 \item $g(x,D_x)$ is a nonsmooth singular Green operator on $\overline{\R_+^n}$ with $C^\tau - C^\infty$--kernel, i.e., there exists $K_g$ such that
 \[
  g(x,D_x)f(x)=\int_{\R^{n-1}}\int_{\R_+}K_g(x,x'-w',x_n,w_n)\,f(w',w_n)\,dw_n\,dw',
 \]
 for all $x\in{\overline{\R_+^n}}$ and for any $f\in\mathcal{S}(\overline{\R_+^n})$, where for all $\alpha,\beta\in\N^{n-1}$, for all $k,k',l,l'\in\N$, there is some constant $C_{\alpha,\beta,k,k',l,l'}>0$ such that
 \[
  \sup_{\substack{z'\in{\R}^{n-1}\\ y_n,w_n\geq0}}\left\|(z')^\beta\partial_{z'}^\alpha y_n^k\partial_{y_n}^{k'} w_n^l\partial_{w_n}^{l'} K_g(\cdot,z',y_n,w_n)\right\|_{C^\tau_{x'}(\R^{n-1})}\leq C_{\alpha,\beta,k,k',l,l'};
 \]
 \item $t(x',D_{x})$ is a nonsmooth trace operator  from $\overline{\R_+^n}$ to ${\R}^{n-1}$ with $C^\tau - C^\infty$--kernel, i.e., there exists $K_t$ such that
 \[
  t(x',D_{x})f(x')=\int_{\R^{n-1}}\int_{\R_+}K_t(x',x'-y',y_n)\,f(y',y_n)\,dy_n\,dy',
 \]
 for all $x'\in{{\R}^{n-1}}$ and for any $f\in\mathcal{S}(\overline{\R_+^n})$, where for all $\alpha,\beta\in\N^{n-1}$, for all $l,l'\in\N$, there is some constant $C_{\alpha,\beta,l,l'}>0$ such that
 \[
  \sup_{\substack{z'\in{\R}^{n-1}\\ y_n\geq0}}\left\|(z')^\beta\partial_{z'}^\alpha y_n^l\partial_{y_n}^{l'} K_t(\cdot,z',y_n)\right\|_{C^\tau_{x'}(\R^{n-1})}\leq C_{\alpha,\beta,l,l'};
 \]
 \item $k(x,D_{x'})$ is a nonsmooth Poisson operator  from ${\R}^{n-1}$ to $\overline{\R_+^n}$, with $C^\tau - C^\infty$--kernel, i.e., there exists $K_k$ such that
 \[
  k(x,D_{x'})f(x)=\int_{\R^{n-1}}K_k(x,x'-y',x_n)\,f(y')\,dy',
 \]
 for all $x\in{\overline{\R_+^n}}$, for any $f\in\mathcal{S}({\R}^{n-1})$ and $K_k$ satisfies the same estimates as $K_t$ before with $C^\tau_{x'}(\R^{n-1})$ replaced by $C^\tau_{x}(\overline{\R^n_+})$.
 \item $s(x',D_{x'})$ is a nonsmooth pseudodifferential operator on $\R^{n-1}$  with $C^\tau - C^\infty$--kernel $K_s$ on ${\R}^{n-1}\times{\R}^{n-1}$ as in Definition \ref{def:Ctau Cinfty kernel PDOs}.
 \end{enumerate}
\end{defin}

\begin{lemma}
\label{lem:green op order -infty has CtauCinfty kernel}
 A Green operator $a(x,D_x)=\twobytwo{p(x,D_x)_++g(x,D_x)}{k(x,D_{x'})}{t(x',D_x)}{s(x',D_{x'})}$ with $C^\tau$--coefficients and of class $0$ is of order $-\infty$ if and only if $a(x,D_x)$ has a $C^\tau - C^\infty$--kernel.
\end{lemma}

\begin{proof}
For singular Green operators, the proof is an adaptation of the proof of Lemma \ref{lem:pdo order -infty has CtauCinfty kernel}, including the new variables $y_n,w_n$, as well as the factor  $y_n^k\partial_{y_n}^{k'} w_n^l\partial_{w_n}^{l'}$ in the kernel estimates, in a straightforward manner. For the convenience of the reader we give the details for this case.\\

We will use the shorthand notation $Y_n:=(y_n,w_n)$ and $Y_n^{kl}\partial_{Y_n}^{k'l'}:=y_n^k\partial_{y_n}^{k'} w_n^l\partial_{w_n}^{l'}$. Integrations with respect to $\xi'$ are understood to be over $\R^{n-1}$.\\

Let $g(x,D_x)$ be a singular Green operator on $\overline{\R_+^n}$ of order $-\infty$ with $C^\tau$--coefficients and class $0$, i.e., the symbol--kernel $g(x,\xi',y_n,w_n)$ belongs to the class $C^\tau S^{-\infty}_{1,0}(\overline{\R^n_+}\times\R^{n-1};\mathcal{S}_{++})$. Define 
\begin{equation}
\label{eq:kernel sgo}
 K_g(x,z',y_n,w_n):=\int_{\R^{n-1}}e^{iz'\xi'}g(x,\xi',y_n,w_n)\,\dbar\xi'
\end{equation}
for all $x\in \overline{\R^n_+}, z'\in \R^{n-1},y_n,w_n\in \R_+$.
The estimate 
 \[
  \left\|(z')^\beta\partial_{z'}^\alpha Y_n^{kl}\partial_{Y_n}^{k'l'} K_g(\cdot,z',\cdot,\cdot)\right\|_{C^\tau_{x}\left(\overline{\R_+^n};{L^2_{Y_n}({\R}_{++}^2)}\right)}\leq C_{\alpha,\beta,k,k',l,l'} 
 \]
can be derived from
\[
 (z')^\beta\partial_{x}^\gamma\partial_{z'}^\alpha K_g(x,z',Y_n)=\int_{\R^{n-1}}e^{iz'\xi'}(i\partial_{\xi'})^\beta\Big((i{\xi'})^\alpha\partial_{x'}^\gamma g(x,\xi',Y_n)\Big)\,\dbar\xi',
\]
which follows from integration by parts and the identity $(z')^\beta e^{iz'\xi'}=(-i\partial_{\xi'})^\beta e^{iz'\xi'}$.\\
Indeed, for all $|\gamma|\leq[\tau]$
\begin{align*}
 &\underset{x\in{\overline{\R_+^n}}}{\sup}\left|(z')^\beta\partial_{z'}^\alpha Y_n^{kl}\partial_{Y_n}^{k'l'}\partial_x^\gamma K_g(x,z',Y_n)\right|\\
 &\leq C\,\underset{x\in\overline{\R^n_+}}{\sup}\int \left|Y_n^{kl}\partial_{Y_n}^{k'l'} D_{\xi'}^\beta\Big({\xi'}^\alpha D_{x}^\gamma g(x,\xi',Y_n)\Big)\right|\,\dbar\xi'\\
 &\leq C\,\sum_{|\delta|\leq|\beta|}\int \left|D_{\xi'}^\delta({\xi'}^\alpha)\right|\left\|Y_n^{kl}\partial_{Y_n}^{k'l'}D_{\xi'}^{\beta-\delta}D_{x}^\gamma g(\cdot,\xi',\cdot,\cdot)\right\|_{L^\infty_x\left(\overline{\R_+^n};{L^2_{Y_n}({\R}_{++}^2)}\right)}\,\dbar\xi'\\
 &\leq C\int_{\R^{n-1}}\langle\xi'\rangle^{-N}\,\dbar\xi'\leq C
\end{align*}
independently of $z'\in\R^{n-1}$ for any $N\geq n$. Therefore, for all $|\gamma|\leq[\tau]$ there exists a constant $C$ such that
\[
 \left\|(z')^\beta\partial_{z'}^\alpha Y_n^{kl}\partial_{Y_n}^{k'l'}\partial_x^\gamma  K_g(\cdot,z',\cdot)\right\|_{L^\infty_{x}(\overline{\R_+^n};{L^2_{Y_n}({\R}_{++}^2)})}\leq C\qquad \text{for all }z'\in \R^{n-1}.
\]
We also have for $|\gamma|=[\tau]$
\begin{align*}
 &\underset{x^0\not=x^1}{\sup}\dfrac{\left|(z')^\beta\partial_{z'}^\alpha Y_n^{kl}\partial_{Y_n}^{k'l'}\partial_x^\gamma K_g(x^0,z',Y_n)-(z')^\beta\partial_{z'}^\alpha Y_n^{kl}\partial_{Y_n}^{k'l'}\partial_x^\gamma K_g(x^1,z',Y_n)\right|}{|x^0-x^1|^{\tau-[\tau]}} \\
 & \leq C\int \underset{x^0\not=x^1}{\sup}\dfrac{\left|Y_n^{kl}\partial_{Y_n}^{k'l'} D_{\xi'}^\beta{\xi'}^\alpha D_{x}^\gamma\Big( g(x^0,\xi',Y_n)- g(x^1,\xi',Y_n)\Big)\right|}{|x^0-x^1|^{\tau-[\tau]}}\,\dbar\xi'\\
 & \leq C\sum_{|\delta|\leq|\beta|}\int \left|D_{\xi'}^\delta({\xi'}^\alpha)\right|\left\|Y_n^{kl}\partial_{Y_n}^{k'l'} D_{\xi'}^{\beta-\delta} g(\cdot,\xi',\cdot,\cdot)\right\|_{C^\tau_x\left(\overline{\R_+^n};{L^2_{Y_n}({\R}_{++}^2)}\right)}\,\dbar\xi'\\
 & \leq C\int_{\R^{n-1}}\langle\xi'\rangle^{-N}\,\dbar\xi' 
 \leq C
\end{align*}
independently of $z'\in \R^{n-1}$ for any $N> n$. Therefore, for all $\alpha,\beta,k,k',l,l'$ there exists a constant $C_{\alpha,\beta,k,k',l,l'}$ such that
\[
 \left\|(z')^\beta\partial_{z'}^\alpha y_n^k\partial_{y_n}^{k'} w_n^l\partial_{w_n}^{l'} K_g(\cdot,z',\cdot,\cdot)\right\|_{C^\tau_{x}\left(\overline{\R_+^n};{L^2_{y_n,w_n}({\R}_{++}^2)}\right)}\leq C_{\alpha,\beta,k,k',l,l'}.
\]
Thus, $K_g$ is a $C^\tau - C^\infty$--kernel of the operator $g(x,D_x)$.\\
For the converse, we apply the Fourier transform with respect to $z'$ on both sides of \eqref{eq:kernel sgo}, so starting from a kernel $K_g(x,\xi',y_n,w_n)$ we get a symbol--kernel
\[
 g(x,\xi',y_n,w_n):=\int_{\R^{n-1}}e^{-iz'\xi'}K_g(x,z',y_n,w_n)\,\dbar z',
\]
and since $K_g\in C^\tau_{x}(\overline{\R_+^n};\mathcal{S}(\R^{n-1}_{z'}\times{\overline{\R_+}}_{y_n}\times{\overline{\R_+}}_{w_n}))$, then its Fourier transform with respect to $z'$ belongs to $C^\tau_{x}(\overline{\R_+^n};\mathcal{S}(\R^{n-1}_{\xi'}\times{\overline{\R_+}}_{,y_n}\times{\overline{\R_+}}_{,w_n}))$, i.e.\ $g(x,\xi',y_n,w_n)\in C^\tau S^{-\infty}_{1,0}(\overline{\R_+^n}\times\R^{n-1};\mathcal{S}_{++})$.\\

For Poisson and trace operators, the proof is very similar to the proof for singular Green operators, just considering $l=l'=0$ and $L^2_{y_n}({\R}_{+})$ instead of $L^2_{y_n,w_n}({\R}_{++}^2)$ in the kernel estimates.
\end{proof}

As before we use the notation $U:=\overline{\R_+^n}$ if $h$ is the symbol--kernel of a Poisson operator, and $U:=\R^{n-1}$ if $h$ is the symbol--kernel of a trace operator.

\begin{note}
 Note that in Definition \ref{def:poisson trace Ctau}, for a symbol--kernel $h\in C^\tau S^m_{1,0}(U\times\R^{n-1};\mathcal{S}_{+})$, the inequality given in \eqref{estimation2poissontrace}
\begin{equation*}
 \left\lVert y_n^l\partial_{y_n}^{l'}\partial_{\xi'}^\alpha h(\cdot,\xi',\cdot)\right\lVert_{C^\tau(U;L_{y_n}^2(\R_+))}\leq C_{\alpha,l,l'}\langle\xi'\rangle^{m+\frac{1}{2}-l+l'-|\alpha|}.
\end{equation*}
 for all $l,l'\in\N$, is equivalent to the inequality
\begin{equation*}
 \underset{y_n\geq0}{\sup}\left\lVert y_n^l\partial_{y_n}^{l'}\partial_{\xi'}^\alpha h(\cdot,\xi',y_n)\right\lVert_{C^\tau(U)}\leq \tilde{C}_{\alpha,l,l'}\langle\xi'\rangle^{m+1-l+l'-|\alpha|}.
\end{equation*}
for all $l,l'\in\N$. Similarly, note that in Definition \ref{def:sgo Ctau}, for a symbol--kernel $g\in C^\tau S^m_{1,0}(\overline{\R_+^n}\times\R^{n-1};\mathcal{S}_{++})$, the inequality given in \eqref{estimation2sgo}
\begin{align*}
& \left\lVert y_n^k\partial_{y_n}^{k'}w_n^l\partial_{w_n}^{l'}\partial_{\xi'}^\alpha g(\cdot,\xi',\cdot,\cdot)\right\lVert_{C^\tau(\overline{\R_+^n};L^2_{y_n,w_n}({\R}_{++}^2))}
 \leq C_{\alpha,k,k',l,l'}\langle\xi'\rangle^{m+1-k+k'-l+l'-|\alpha|}
\end{align*}
 is equivalent to the inequality
\begin{align*}
& \underset{y_n,w_n\geq0}{\sup}\left\lVert y_n^k\partial_{y_n}^{k'}w_n^l\partial_{w_n}^{l'}\partial_{\xi'}^\alpha g(\cdot,\xi',y_n,w_n)\right\lVert_{C^\tau(\overline{\R_+^n})} 
\leq C_{\alpha,k,k',l,l'}\langle\xi'\rangle^{m+2-k+k'-l+l'-|\alpha|}. 
\end{align*}
This follows from the inequalities (see \cite[Lemma 4.6]{NonsmoothGreen})
\begin{align*}
 \|f\|_{L^\infty(\R_+)}&\leq C\|f\|_{L^2(\R_+)}^{\frac{1}{2}}\|\partial_{y_n}f\|_{L^2(\R_+)}^{\frac{1}{2}}, \\
 \|f\|_{L^2(\R_+)}&\leq C\|f\|_{L^\infty(\R_+)}^{\frac{1}{2}}\|y_nf\|_{L^\infty(\R_+)}^{\frac{1}{2}}.
\end{align*}
\end{note}

\begin{lemma}
\label{lem:eta k has Ctau Cinfty kernel}
 Let $\eta\in C_{b}^\infty(\overline{\R_+})$ be such that $\eta\equiv0$ on $[0,\delta]$, and $\eta\equiv1$ on $[2\delta,+\infty)$ for some $\delta>0$. Let ${h}\in C^\tau S^{m}_{1,0}(U\times\R^{n-1};\mathcal{S}_+)$ be the symbol--kernel of a Poisson or trace operator with $C^\tau$--coefficients. Then, $\eta(x_n)h(x,D_x)$ has a $C^\tau - C^\infty$--kernel.
\end{lemma}

\begin{proof}
 First of all, observe that ${h}(\cdot,\cdot,y_n)\in C^\tau S^{-\infty}_{1,0}(U\times\R^{n-1})$ uniformly in $y_n\geq\delta$ for every $\delta>0$. Indeed, if ${h}\in C^\tau S^{m}_{1,0}(U\times\R^{n-1};\mathcal{S}_+)$, then for all $\alpha\in\N^{n-1}$, for all $N,l,l'\in\N$, there exists a constant $C$ such that for all $\xi'\in\R^{n-1}$
\[
 \underset{y_n\geq0}{\sup}\left\|y_n^{l+N}\partial_{y_n}^{l'}\partial_{\xi'}^\alpha{h}(\cdot,\xi',y_n)\right\|_{C^\tau(U)} \leq C\langle\xi'\rangle^{m+1-N-l+l'-|\alpha|},
\]
and, if $y_n\geq\delta>0$, this implies that
\begin{align*}
 \left\|y_n^l\partial_{y_n}^{l'}\partial_{\xi'}^\alpha{h}(\cdot,\xi',y_n)\right\|_{C^\tau(U)} 
 &\leq Cy_n^{-N}\langle\xi'\rangle^{m+1-N-l+l'-|\alpha|}
 \leq C_{\delta}\langle\xi'\rangle^{m+1-N-l+l'-|\alpha|},
\end{align*}
for some constant $C_{\delta}$. Hence for all $N,l,l'\in\N$ and $y_n\geq \delta>0$
\[
 \left\|y_n^l\partial_{y_n}^{l'}\eta(y_n)\partial_{\xi'}^\alpha{h}(\cdot,\xi',y_n)\right\|_{C^\tau(U)}\leq C_{\delta}\langle\xi'\rangle^{m+1-N-l+l'-|\alpha|}.
\]
Therefore $\eta(y_n){h}(x,\xi',y_n)\in C^\tau S^{-\infty}_{1,0}(U\times\R^{n-1};\mathcal{S}_+)$, which by Lemma \ref{lem:green op order -infty has CtauCinfty kernel} implies that the operator $\eta(x_n){h}(x,D_x)$ has a $C^\tau - C^\infty$--kernel.
\end{proof}

In the following sections, given two sets $A,B\subseteq\overline{\R_+^n}$ we denote by $\dist(A,B)$ the distance between them, i.e.\ $\dist(A,B):=\inf\{|x-y|:x\in A, y\in B\}$\index{$\dist(A,B)$}, where $|\cdot|$ denotes the Euclidean distance in $\overline{\R_+^n}$.

The transformation of an operator on a manifold with boundary after a coordinate change, produces operators with $C^\tau - C^\infty$--kernel, and in the following sections, we study these operators, which are localizations of operators acting on and going to the boundary of $\overline{\R_+^n}$.

\subsection{Nonsmooth Poisson Operators}

\begin{corol}
\label{cor:phi k(x,Dx) CtauCinfty kernel}
 Let $\varphi\in C_{b}^\infty(\overline{\R_+^n})$ be such that $\dist(\supp\varphi,\partial\overline{\R_+^n})>0$. Let ${k}\in C^\tau S^{m-1}_{1,0}(\overline{\R^n_+}\times\R^{n-1};\mathcal{S}_+)$ be the symbol--kernel of a Poisson operator with $C^\tau$--coefficients. Then the operator $\varphi(x){k}(x,D_{x'})$ has a $C^\tau - C^\infty$--kernel.
\end{corol}

\begin{proof}
 Let $\delta>0$ be such that $\dist(\supp\varphi,\partial\overline{\R_+^n})=\delta$. Let $\eta\in C_{b}^\infty(\overline{\R_+})$ be such that $\eta\equiv0$ on $[0,\delta/2]$, and $\eta\equiv1$ on $[\delta,+\infty)$. From Lemma \ref{lem:eta k has Ctau Cinfty kernel} the operator $\eta(x_n)k(x,D_{x'})$ has a $C^\tau - C^\infty$--kernel. Hence the product by the smooth function $\varphi$, $\varphi(x){k}(x,D_{x'})=\varphi(x)\eta(x_n){k}(x,D_{x'})$ has a $C^\tau - C^\infty$--kernel.
\end{proof}

\begin{prop}
\label{prop:phi k(x,Dx)psi CtauCinfty kernel Poisson}
Let $\varphi\in C_{b}^\infty(\overline{\R_+^n})$, $\psi\in C_{b}^\infty({\R}^{n-1})$ be such that $\supp\varphi\cap(\supp\psi\times \{0\})=\emptyset$. Let ${k}\in C^\tau S^{m-1}_{1,0}(\overline{\R^n_+}\times\R^{n-1};\mathcal{S}_+)$ be the symbol--kernel of a Poisson operator with $C^\tau$--coefficients. Then the operator $\varphi(x){k}(x,D_{x'})\psi(x')$ has a $C^\tau - C^\infty$--kernel.
\end{prop}

\begin{proof}
 Choose $\delta>0$ so small that 
\[
 P_{\partial\overline{\R_+^n}}\big(\supp\varphi \cap({\R}^{n-1}\times[0,\delta])\big)\cap\supp\psi=\emptyset,
\]
where $P_{\partial\overline{\R_+^n}}$ denotes the orthogonal projection of $\overline{\R_+^n}$ onto ${\partial\overline{\R_+^n}}$
Set $A:=P_{\partial\overline{\R_+^n}}\big(\supp\varphi \cap({\R}^{n-1}\times[0,\delta])\big)$. Let $\eta\in C_{b}^\infty(\overline{\R_+})$ be such that $\eta\equiv1$ on $[0,\delta/2]$, and $\supp\eta\subseteq[0,\delta]$. One can write
\begin{align*}
 &\varphi(x) k(x,D_{x'})\psi(x')=
(\varphi(x)\eta(x_n))k(x,D_{x'})\psi(x') + \varphi(x)(1-\eta(x_n))k(x,D_{x'})\psi(x').
\end{align*}
Since $1-\eta\equiv0$ on $[0,\delta/2]$, by Corollary \ref{cor:phi k(x,Dx) CtauCinfty kernel} $\varphi(x)(1-\eta(x_n))k(x,D_{x'})$ has a $C^\tau - C^\infty$--kernel and therefore the second term on the right hand side has a $C^\tau - C^\infty$--kernel. Now choose $\widetilde{\psi}\in C_{b}^\infty({\R}^{n-1})$ such that $\widetilde{\psi}\equiv1$ on $A$ and $\supp\widetilde{\psi}\cap\supp\psi=\emptyset$. Then by  Proposition \ref{prop:kernel phi p psi}, the operator 
\begin{align*}
 \widetilde{\psi}(x')k(x,D_{x'})\psi(x')&=\widetilde{\psi}(x')\Op'\left(k(x,\xi',D_n)\right)\psi(x')
=\Op'\left(\widetilde{\psi}(x')k(x,\xi',D_n)\psi(x')\right)
\end{align*}
has a $C^\tau - C^\infty$--kernel. Thus,
\[
 (\varphi(x)\eta(x_n))k(x,D_{x'})\psi(x')=(\varphi(x)\eta(x_n))\widetilde{\psi}(x')k(x,D_{x'})\psi(x')
\]
has a $C^\tau - C^\infty$--kernel because of Lemma~\ref{lem:eta k has Ctau Cinfty kernel}.
\end{proof}

\subsection{Nonsmooth Trace Operators}

\begin{corol}
\label{cor:t(x,Dx)phi CtauCinfty kernel}
Let $\varphi\in C_{b}^\infty(\overline{\R_+^n})$ be such that $\dist(\supp\varphi,\partial\overline{\R_+^n})>0$. Let ${t}(x',\xi',y_n)\in C^\tau S^{m}_{1,0}(\R^{n-1}\times\R^{n-1};\mathcal{S}_+)$ be the symbol--kernel of a trace operator with $C^\tau$--coefficients. Then ${t}(x',D_x)(\varphi\, \cdot)(x')$ has a $C^\tau - C^\infty$--kernel.
\end{corol}

\begin{proof}
 Because of \eqref{eq:higher class trace operators}
\[
 {t}(x',D_x)(\varphi)(x')={t}_0(x',D_x)(\varphi)(x')+\sum_{j=0}^{r-1}s_j(x',D_{x'})(\gamma_j\varphi)(x'),
\]
and since all the terms in the sum of the right hand side vanish, we can assume without loss of generality that ${t}(x',D_x)$ is of class 0. Then we have for all $f\in\mathcal{S}(\overline{\R_+^n})$
\begin{equation*}
 {t}(x',D_x)(\varphi f)(x')=\mathcal{F}^{-1}_{\xi'\mapsto x'}\left[\int_0^\infty{t}(x',\xi',y_n)\mathcal{F}_{x'\mapsto\xi'}[\varphi f(\cdot,y_n)]\,dy_n\right].
\end{equation*}
As before, since there exists $\delta>0$ such that $\mathcal{F}_{x'\mapsto\xi'}[\varphi f(\cdot,y_n)]=0$ if $y_n\leq\delta$, choosing $\eta\in C_{b}^\infty(\overline{\R_+})$ such that $\eta\equiv0$ if $|y_n|\leq\frac{\delta}{2}$, and $\eta\equiv1$ if $|y_n|\geq\delta$, we can write the right hand side of the previous equation as 
\begin{equation*}
 \mathcal{F}^{-1}_{\xi'\mapsto x'}\left[\int_0^\infty{t}(x',\xi',y_n)\eta(y_n)\mathcal{F}_{x'\mapsto\xi'}[\varphi f(\cdot,y_n)]\,dy_n\right]={t}_1(x',D_x)(\varphi f)(x').
\end{equation*}
where $t_1(x',\xi',y_n):={t}(x',\xi',y_n)\eta(y_n)$, and as in Lemma \ref{lem:eta k has Ctau Cinfty kernel} one shows that $t_1(x',D_x)$ has a $C^\tau - C^\infty$--kernel and therefore ${t}(x',D_x)\varphi$ has a $C^\tau - C^\infty$--kernel.
\end{proof}

\begin{prop}
\label{prop:phi k(x,Dx)psi CtauCinfty kernel trace}
Let $\varphi\in C_{b}^\infty(\overline{\R_+^n})$, $\psi\in C_{b}^\infty({\R}^{n-1})$ be such that $\supp\varphi\cap(\supp\psi\times \{0\})=\emptyset$. Let ${t}\in C^\tau S^{m}_{1,0}(\R^{n-1}\times\R^{n-1};\mathcal{S}_+)$ be the symbol--kernel of a trace operator with $C^\tau$--coefficients. Then the operator $\psi(x'){t}(x',D_x)\varphi(x)$ has a $C^\tau - C^\infty$--kernel.
\end{prop}

\begin{proof}
We choose $\delta>0$ so small that 
\[
 P_{\partial\overline{\R_+^n}}\big(\supp\varphi \cap({\R}^{n-1}\times[0,\delta])\big)\cap\supp\psi=\emptyset,
\]
where as before $P_{\partial\overline{\R_+^n}}$ denotes the orthogonal projection of $\overline{\R_+^n}$ onto ${\partial\overline{\R_+^n}}$.
Moreover, we set $A:=P_{\partial\overline{\R_+^n}}\big(\supp\varphi \cap({\R}^{n-1}\times[0,\delta])\big)$. Let $\eta\in C_{b}^\infty(\overline{\R_+})$ be such that $\eta\equiv1$ on $[0,\delta/2]$, and $\supp\eta\subseteq[0,\delta]$. Then
\begin{align*}
 &\psi(x') t(x',D_x)\varphi(x)=
 \psi(x') t(x',D_x)(\eta(x_n)\varphi(x)) + \psi(x') t(x',D_x)((1-\eta(x_n))\varphi(x)).
\end{align*}
Since $1-\eta\equiv0$ on $[0,\delta/2]$,  the second term on the right hand side has a $C^\tau - C^\infty$--kernel by the same arguments as in Corollary \ref{cor:t(x,Dx)phi CtauCinfty kernel}. As in Proposition \ref{prop:phi k(x,Dx)psi CtauCinfty kernel Poisson}, choosing $\widetilde{\psi}\in C_{b}^\infty({\R}^{n-1})$ such that $\widetilde{\psi}\equiv1$ on $A$ and $\supp\widetilde{\psi}\cap\supp\psi=\emptyset$, the first term can be written as $\psi(x') t(x',D_x)\widetilde{\psi}(x')(\eta(x_n)\varphi(x))$ which also has a $C^\tau - C^\infty$--kernel because of Proposition \ref{prop:kernel phi p psi}.
\end{proof}

\subsection{Nonsmooth Singular Green Operators}

\begin{lemma}\label{lem:CtauInftyKernelG}
Let $\varphi\in C_{b}^\infty(\overline{\R_+^n})$ be such that $\dist(\supp\varphi,\partial\overline{\R_+^n})>0$, and let ${g}\in C^\tau S^{m}_{1,0}(\overline{\R_+^n}\times\R^{n-1};\mathcal{S}_{++})$ be the symbol--kernel of a singular Green operator with $C^\tau$--coefficients $g(x,D_x)$. Then $\varphi(x){g}(x,D_x)$ and ${g}(x,D_x)\varphi(x)$  have a $C^\tau - C^\infty$--kernel.
\end{lemma}
\begin{proof}
The statements are proven in the same way as in the proofs of  Corollary \ref{cor:phi k(x,Dx) CtauCinfty kernel} and Corollary \ref{cor:t(x,Dx)phi CtauCinfty kernel}.
\end{proof}

\begin{prop}
\label{prop:phi k(x,Dx)psi CtauCinfty kernel sgo}
 Let $\varphi,\psi\in C_{b}^\infty(\overline{\R_+^n})$ be such that $\supp\varphi\cap\supp\psi\cap (\R^{n-1}\times \{0\})=\emptyset$. Let ${g}\in C^\tau S^{m}_{1,0}(\overline{\R_+^n}\times\R^{n-1};\mathcal{S}_{++})$ be the symbol--kernel of a singular Green operator with $C^\tau$--coefficients. Then the operator $\varphi(x){g}(x,D_x)\psi(x)$ has a $C^\tau - C^\infty$--kernel.
\end{prop}

\begin{proof}
 Let $\delta>0$ be sufficiently small and let $\eta\in C_{b}^\infty(\overline{\R_+})$ be such that $\eta\equiv1$ on $[0,\delta/2]$,  $\supp\eta\subseteq[0,\delta]$ and
\[
 P_{\partial\overline{\R_+^n}}(\supp(\varphi\eta))\cap P_{\partial\overline{\R_+^n}}(\supp(\eta\psi))=\emptyset.
\]
One can write
\begin{align*}
 &\varphi(x) g(x,D_x)\psi(x)
 =\varphi(x)\eta(x_n) g(x,D_x)\psi(x) + \varphi(x)(1-\eta(x_n))g(x,D_x)\psi(x) \\
 &\ =\varphi(x)\eta(x_n) g(x,D_x)(\eta(x_n)\psi(x)) + \varphi(x)\eta(x_n) g(x,D_x)(1-\eta(x_n))\psi(x) \\
 &\ \ \ + \varphi(x)(1-\eta(x_n))g(x,D_x)\psi(x).
\end{align*}
Each term in this sum has a $C^\tau - C^\infty$--kernel because of Proposition \ref{prop:kernel phi p psi} and Lemma~\ref{lem:CtauInftyKernelG}, respectively.
\end{proof}

\section{Transmission Condition}
\label{section:transmission condition}

The transmission condition is a condition on operators which allows to preserve regularity up to the boundary. 
In \cite[Definition 5.2]{NonsmoothGreen} a transmission condition for nonsmooth pseudodifferential operators is defined. See also Remark 5.3 in loc.\ cit.\ for a comparison with the definition given in \cite[Definition 2.2.4.]{FunctionalCalculus}, which is the one we use here.

\begin{defin}\label{defn:TransmissionCond}
Let $p\in C^\tau S^m_{1,0}(\Rn\times\Rn)$, $m\in\Z$.  Then $p$ satisfies the \emph{global transmission condition} {at $x_n=0$} -- simply called transmission condition in the following -- if there are functions $s_{k,\alpha}(x',\xi')$ smooth in $\xi'$  and in $C^\tau$ with respect to $x'$ such that for any $\alpha\in\N^n$ and $l\in \N$
\begin{equation}\label{eq:TransmissionCond}
  \left\|\xi_n^l D_{\xi}^{\alpha}p(\cdot,0,\xi)-\sum_{k=-l}^{m-|\alpha|} s_{k,\alpha} (\cdot,\xi')\xi_n^{k+l}\right\|_{C^\tau(\R^{n-1})}
  \leq C_{l,\alpha}
  \langle\xi'\rangle^{m+l+1-|\alpha|}|\xi_n|^{-1}
\end{equation}
{for some constants $C_{l,\alpha}$,} when $|\xi_n|\geq \langle\xi'\rangle$.
\end{defin}

The non--smooth transmission condition ensures natural mapping properties for the associated truncated pseudodifferential operators and a good behaviour under compositions \cite{NonsmoothGreen}.

In the following we show that the transmission condition is preserved under a suitable smooth coordinate change. We follow the strategy in the smooth case in \cite[Section 2.2]{FunctionalCalculus} and refer to that book for an introduction and results on the transmission condition in the smooth case.

For the following coordinate change we need $(x,y)$-symbols, also known as double symbols, which are nonsmooth in $x$ and smooth in $y$ similarly as in \cite{AbelsPfeufferCharacterization,Diss}:
\begin{defin}
Let $0<s <1$, $0\leq \rho \leq 1$, $k \in \N$ and $m\in \R$. Then $C^{k+s} S^m_{\rho, 0}(\Rn\times \Rn \times \Rn)$ is the set of all functions $p\colon\Rn \times\Rn \times \Rn \to \C$ such that
\begin{itemize}
\item[(i)] $\partial_\xi^{\alpha}  \partial^{\delta}_{y} p(\cdot,y,\xi) \in C^{k+s}_x(\R^n)$ and $\partial_x^{\beta} \partial_\xi^{\alpha} \partial^{\delta}_{y} p \in C^{0}(\R^n \times \R^n \times \R^n)$,
\item[(ii)] $\left\|\partial_\xi^{\alpha}  \partial^{\delta}_{y} p(\cdot,y,\xi)  \right\|_{C^{k+s}(\R^n)} \leq C_{\alpha, \delta} \langle\xi\rangle^{m-\rho|\alpha|}$ 
\end{itemize}
for all $y,\xi \in \Rn$ and arbitrary $\alpha, \beta, \delta \in \N^n$ with $|\beta| \leq k$. Here the constant $C_{\alpha,\delta}$ is independent of $y,\xi \in \Rn$. We also define the operator
\[
p(x,D_x,x)u(x):=\int_{\R^n}\int_{\R^n} e^{i(x-y)\xi}{p(x,y,\xi)}u(y)\,dy\,\dbar\xi,
\]
for all $u\in\mathcal{S}(\R^n)$, and
\begin{equation}
\label{eq:pL}
p_L(x,\xi):= {\rm Os-}\iint e^{-iy\cdot \eta} p(x,x+y,\xi+\eta)\, dy \, \dbar\eta.
\end{equation}
\end{defin}

\begin{thm}\label{thm:reduced left symbol}
Let $\tau>0$, $\tau \not \in\N$, and let $p \in C^\tau S^m_{1, 0}(\Rn\times \Rn \times \Rn)$. Then 
\begin{equation}\label{eq:Reduction}
 p_L(x,D_x)u= p(x,D_x,x) u\qquad \text{for every } u\in \mathcal{S}(\Rn),
\end{equation}
where $p_L\in C^\tau S^m_{1, 0}(\Rn\times \Rn)$\index{$p_L$} and 
\begin{equation}\label{eq:AsympExpansion}
 p_L(x,\xi)\sim \sum_{\alpha\in \N^n} \left.\frac1{\alpha!} \partial_\xi^\alpha D_y^\alpha p(x,y,\xi)\right|_{y=x}.
\end{equation}
\end{thm}

\begin{proof}
First of all \cite[Theorem~4.15]{AbelsPfeufferCharacterization} implies that \eqref{eq:Reduction} holds true, where $p_L \in C^\tau S^m_{0,0}(\Rn\times \Rn)$. Moreover, because of \cite[Theorem~2.11]{AbelsPfeufferCharacterization}, 
\begin{equation*}
  \partial_\xi^\alpha p_L(x,\xi) = {\rm Os-}\iint e^{-iy\cdot \eta} \partial_\xi^{\alpha}p(x,x+y,\xi+\eta)\, dy \,\dbar\eta
\end{equation*}
for any $\alpha\in\N^n$, where $\partial_\xi^\alpha p\in C^\tau S^{m-|\alpha|}_{1,0}(\Rn\times \Rn\times \Rn)$. Hence applying \cite[Theorem~4.15]{AbelsPfeufferCharacterization} again yields that $\partial_\xi^\alpha p_L\in  C^\tau S^{m-|\alpha|}_{0,0}(\Rn\times \Rn)$ for every $\alpha\in\N^n$, which means that $p_L\in C^\tau S^m_{1,0}(\Rn\times \Rn)$. 

Finally, the asymptotic expansion \eqref{eq:AsympExpansion} is shown in the same way as for smooth double symbols $p\in S^m_{1,0}(\Rn\times \Rn\times \Rn)$ by using a Taylor series expansion of the integrand in \eqref{eq:pL} with respect to $x+y$ and integration by parts. Here the nonsmoothness of $p$ with respect to the first variable does not matter since only derivatives with respect to the second and third variables are involved and $p$ is smooth with respect to these variables. Moreover, also the oscillatory integral is only taken with respect to these variables. In this way for any $N\in\N$ we have
\begin{equation*}
  p_L(x,\xi)= \sum_{|\alpha|\leq N}\left. \frac1{\alpha!} \partial_\xi^\alpha D_y^\alpha p(x,y,\xi)\right|_{y=x} + r_N(x,\xi),
\end{equation*}
where
\begin{equation*}
  r_N(x,\xi)= (N+1)\sum_{|\alpha|\leq N+1} \frac1{\alpha!}{\rm Os-}\iint e^{-iy\cdot \eta} D_y^\alpha r_\alpha(x,\xi,y,\eta)\, dy \, \dbar\eta
\end{equation*}
and
\begin{equation*}
  r_\alpha (x,\xi,y,\eta)= \int_0^1 \partial_\xi^\alpha p(x,x+y,\xi+\theta\eta)(1-\theta)^N \, d\theta,
\end{equation*}
cf.\ e.g.\ \cite[Proof of Theorem 3.15]{Abels:PDOs}. Now one can estimate $r_\alpha$ in a straight forward manner (as in the smooth case) and derive together with \cite[Theorems~2.11 and 4.15]{AbelsPfeufferCharacterization} that $r_N\in C^\tau S^{m-N-1}_{1,0}(\Rn\times \Rn)$.
\end{proof}

For the following let $\kappa:\R^n\to \R^n$ be a bounded smooth diffeomorphism, $\kappa:=(\kappa_1,\ldots,\kappa_n)$, that preserves the sets $\partial \overline{\R^n_+}$ and $\R^n_+$, i.e, they are mapped to itself. 
In Theorem \ref{theorem1} we have seen that given $p\in C^\tau S^m_{1,0}(\Rn\times\Rn)$, the symbol of the operator given by Equation \eqref{eq:change of variables}
\begin{equation}\label{eq:transformPsDO}
 \widetilde{p}(x,D_x):=\kappa^{-1,*}p(x,D_x)\kappa^*
\end{equation}
belongs to $C^\tau S^m_{1,0}(\Rn\times\Rn)$. Moreover, in \cite[Theorem 2.2.13]{FunctionalCalculus} it is proven that in the case of operators with smooth coefficients, if the symbol $p$ of the operator $p(x,D_x)$ satisfies the transmission condition, then the symbol $\widetilde{p}$ of the operator $\widetilde{p}(x,D_x)$ also satisfies the transmission condition, i.e.\ the transmission condition is preserved under a smooth coordinate change. Now we will show that this proof can be adapted to the nonsmooth case.

\begin{thm}
  Let $\kappa$ be as before, $m\in\Z$ and $\tau>0$, $\tau\not\in\N,$ and $p\in C^\tau S^m_{1,0}(\Rn\times\Rn)$ satisfy the transmission condition at $x_n=0$. Then the transformed symbol $\widetilde{p}$ given by \eqref{eq:transformPsDO} also satisfies the transmission condition at $x_n=0$.
\end{thm}

\begin{proof}
 Let us denote by $\kappa'$ the Jacobian matrix $\left(\partial\kappa_i/\partial x_j\right)$.  For $x,y\in\R^n$  we  set $\underline{x}:=\kappa(x), \underline{y}:=\kappa(y)$ and have 
\begin{equation*}
\underline{x}-\underline{y}=\kappa(x)-\kappa(y)=M(x,y)(x-y),
\end{equation*}
where
\begin{equation*}
M(x,y):=\int_0^1\kappa'(x+t(y-x))\,dt.
\end{equation*}
Here $M(x,x)=\kappa'(x)$ is invertible. Therefore $M(x,y)$ is invertible for $(x,y)$ in a neighborhood of the diagonal $\{(x,y)\in\overline{\R^n_+}\times\overline{\R^n_+}:\ x=y\}$.\\
The fact that $\kappa$ preserves the set $\partial \overline{\R^n_+}$ implies that $\kappa_n(x',0)=0$ for all $x'\in\R^{n-1}$. Then, by Taylor's formula, there exists a $C^\infty$-function $C_1\not=0$, such that
\begin{equation*}
\underline{x}_n=\kappa_n(x',x_n)=\kappa_n(x',0)+x_n\int_0^1\partial_{x_n}\kappa_n(x',hx_n)\,dh=C_1(x)x_n.
\end{equation*}
Since $\partial_{x_j}\kappa_n(x',0)=0$ for $j=1,\ldots,n-1$, the matrix $M(x,y)$ can be written in blocks
\begin{equation*}
M(x,y)=\twobytwo{A(x,y)}{B(x,y)}{B'(x,y)}{C(x,y)},
\end{equation*}
with $A:=(M_{jk})_{1\leq j,k\leq n-1}$, $B:=(M_{jn})_{1\leq j\leq n-1}$, $B':=(M_{nk})_{1\leq k\leq n-1}$, $C:=M_{nn}$, having the following properties for all $x',y'\in\R^{n-1}$:
\begin{enumerate}
\item $A_0(x'):=A(x',0,x',0)=(\partial_{x_k}\kappa_j(x',0))_{j,k\leq n-1}$,
\item $B_0(x'):=B(x',0,x',0)=(\partial_{x_n}\kappa_j(x',0))_{j\leq n-1}$,
\item $B'(x',0,y',0)=0$,
\item $C_0(x'):=C_1(x',0):=C(x',0,x',0)=\partial_{x_n}\kappa_n(x',0)$,
\end{enumerate}
with $A_0(x')$ and $C_0(x')$ invertible with bounded inverses. One also has
\begin{equation*}
\underline{x}'-\underline{y}'=A(x,y)(x'-y')+B(x,y)(x_n-y_n).
\end{equation*}
One can also add to $p$ an operator with $C^\tau - C^\infty$--kernel such that the invertibility holds on $\supp(K_p)$, where $K_p$ denotes the Schwartz kernel of $p(x,D_x)$. Here it is easy to observe that the class of pseudodifferential operators with a  $C^\tau - C^\infty$--kernel are invariant with respect to a smooth coordinate transformations.\\
The relation $\xi=M(x,y)^t\underline{\xi}$ can be written as
\begin{equation*}
\xi=\begin{pmatrix}\xi'\\\xi_n\end{pmatrix}=\begin{pmatrix}{A(x,y)^t\underline{\xi}'}+B'(x,y)^t\underline{\xi}_n\\ {B(x,y)^t\underline{\xi}'}+C(x,y)\underline{\xi}_n\end{pmatrix}
\end{equation*}
and in particular
\begin{align}
\xi'&=A^t\underline{\xi}' \text{ for } (x,y)=(x',0,y',0), \label{xi'}\\
\xi_n&={B^t\underline{\xi}'}+C\underline{\xi}_n. \label{xin}
\end{align}
When this holds and $x'-y'$ is small,
\begin{equation*}
\underline{\xi}=\begin{pmatrix}\underline{\xi}'\\ \underline{\xi}_n\end{pmatrix}=\begin{pmatrix}{(A^t)^{-1}{\xi}'}\\ -\frac{1}{C}{B^t(A^t)^{-1}{\xi}'}+\frac{1}{C}{\xi}_n\end{pmatrix}.
\end{equation*}
Moreover, when we set $\underline{v}(\underline{\xi})=v(\xi)$, we have for the corresponding derivatives:
\begin{align}
D_{\underline{\xi}'}\underline{v}(\underline{\xi})&=\left(\sum_{k=1}^n\frac{\partial\xi_k}{\partial\underline{\xi}_j}D_{\xi_k}v(\xi)\right)_{1\leq j\leq n-1}=AD_{\xi'}v(\xi)+BD_{\xi_n}v(\xi)\label{Dxi'}\\
D_{\underline{\xi}_n}\underline{v}(\underline{\xi})&=CD_{\xi_n}v(\xi). \label{Dxin}
\end{align}

When \eqref{xi'} and \eqref{xin} hold, since $\kappa$ is a bounded smooth diffeomorphism there is a positive function  $R(x', y')$ defined for $x'-y'$ close to zero, so that 
\begin{equation}
\label{eq:constant xi' undelrinexi'}
R(x', y')\langle\xi'\rangle\leq \langle\underline{\xi}'\rangle\leq R^{-1}(x', y')\langle\xi'\rangle.
\end{equation}
Moreover, we have the usual equivalence $\langle\xi\rangle\sim\langle\underline{\xi}\rangle$.\\

For $u\in\mathcal{S}(\R^n)$ we have
\[
(p(x,D_x)(u\circ\kappa))\circ\kappa^{-1}=:p_\kappa(x,D_x,x)u
\]
where 
\[
p_\kappa(\underline{x},\underline{y},\underline{\xi})=p(x,y,M(x,y)^t\underline{\xi}))|\det M(x,y)||\det\kappa'(y)^{-1}|
\]
with the previous notations. Alternatively,
\[
p_\kappa(x,y,\xi)=p(\kappa^{-1}(x),\kappa^{-1}(y),\tilde M(x,y)^t\xi)|\det\tilde M(x,y)||\det D(\kappa^{-1}(y))|
\]
and $\tilde M(x,y):=M(\kappa^{-1}(x),\kappa^{-1}(y))$, where $p_\kappa$ is as in Theorem \ref{thm:change of variables smooth pdos 2}.\\

Now we apply Theorem \ref{thm:reduced left symbol} to $p_\kappa$.  Let $l,N\in\N$. Then there exists $p_L\in C^\tau S^{m}_{1,0}(\R^n\times\R^n)$ such that 
\[
p_\kappa(x,D_x,x)u=p_L(x,D_x)u
\]
where, for all $\alpha\in\N^n$
\[
\xi_n^l\partial_\xi^\alpha p_L(x,\xi)=\xi_n^l\sum_{|\beta|\leq N}\left.\dfrac{1}{\alpha!}\partial_\xi^{\alpha+\beta}D_y^\beta p_\kappa(x,y,\xi)\right|_{y=x}+\xi_n^lr_{N,\alpha}(x,\xi)
\]
and $r_{N,\alpha}\in C^\tau S^{m-N-|\alpha|-1}_{1,0}$ for $N\geq m+l-|\alpha|$. Hence
\begin{equation*}
\left|\xi_n^lr_{N,\alpha}(x,\xi)\right|\leq C\langle\xi\rangle^{m+l-N-1-|\alpha|}\leq C\langle\xi'\rangle^{m+l-|\alpha|+1}|\xi_n|^{-1}.
\end{equation*}
for all $|\xi_n|\geq\langle\xi'\rangle$.
The last inequality follows from the following:
Set $I:=\langle\xi\rangle^{m+l-N-|\alpha|}\langle\xi\rangle^{-1}$. 
In the case $m+l-|\alpha|\geq0$ we choose $N\in\N$ such that $m+l-N-|\alpha|\leq0$. Then
\[
I\leq C|\xi_n|^{-1}\leq C\langle\xi'\rangle^{m+l-|\alpha|+1}|\xi_n|^{-1}.
\]
because of  $\langle\xi\rangle^{-1}\leq|\xi_n|^{-1}$.
In the case $m+l-|\alpha|\leq0$ we  choose $N\geq0$. Then
\[
I\leq C\langle\xi'\rangle^{m+l-N-|\alpha|}|\xi_n|^{-1}\leq C\langle\xi'\rangle^{m+l-|\alpha|+1}|\xi_n|^{-1}.
\]
In the last case we even have
\begin{equation*}
  \left|\xi_n^l \partial_{\xi}^{\alpha} p_L(x',0,\xi)\right|
  \leq C_{l,\alpha}
  \langle\xi\rangle^{m+l-|\alpha|+1}\langle\xi\rangle^{-1}
  \leq C_{l,\alpha}
  \langle\xi'\rangle^{m+l-|\alpha|+1}|\xi_n|^{-1}.
\end{equation*}

\noindent Now observe that 
\begin{equation}
\label{eq:derivatives pkappa in Ctau}
\left.\partial_\xi^{\alpha+\beta}D_y^\beta p_\kappa(x,y,\xi)\right|_{y=x}\in C^\tau S^{m-|\beta|-|\alpha|}_{1,0}(\Rn\times\Rn)
\end{equation}
satisfies the transmission condition. Indeed, for $\alpha=\beta=0$ this holds by the following argument:\\
We must show that the estimate \eqref{eq:TransmissionCond} holds for the transformed symbol $p_\kappa$. We note that when $p$ satisfies the estimate \eqref{eq:TransmissionCond}, and $\xi$ and $\underline{\xi}$ are related by \eqref{xi'}--\eqref{xin}, using \eqref{eq:constant xi' undelrinexi'} we obtain
\begin{align*}
&\left\|\xi_n^l p(\cdot,0,\xi)-\sum_{k=-l}^{m} s_{k} (\cdot,\xi')\xi_n^{k+l}\right\|_{C^\tau(\R^{n-1})} \\
&\leq \left\|({B^t\underline{\xi}'}+C\underline{\xi}_n)^l p(\cdot,0,M^t\underline{\xi})-\sum_{k=-l}^{m} s_{k} (\cdot,A^t\underline{\xi}')({B^t\underline{\xi}'}+C\underline{\xi}_n)^{k+l}\right\|_{C^\tau(\R^{n-1})} \\
&\leq C_{l}\langle\xi'\rangle^{m+l+1}|\xi_n|^{-1} 
\leq C_{l}R^{-d-l-1}\langle\underline{\xi}'\rangle^{m+l+1}|\underline{\xi}_n|^{-1}.
\end{align*}
For small $x'-y'$ (where $C\not=0$), one can use these estimates successively for $l=0,1,\ldots$ to determine expansion coefficients $s'_k(\cdot,\underline{\xi}')$ such that
\begin{equation*}
\left\|\underline{\xi}_n^l p(\cdot,0,M^t\underline{\xi})-\sum_{k=-l}^{m} s'_{k} (\cdot,\underline{\xi}')\underline{\xi}_n^{k+l}\right\|_{C^\tau(\R^{n-1})} \leq R_l\langle\underline{\xi}'\rangle^{m+l+1}|\underline{\xi}_n|^{-1}
\end{equation*}
for some constant $R_l$, for all $l$. Here the $s'_k$ are polynomials in $\underline{\xi}'$, since the $s_k$ are polynomials in $\xi'$. \\

Now, for $\alpha\not=0$ or $\beta\not=0$, $\partial_\xi^{\alpha+\beta}D_y^\alpha p_\kappa(x,y,\xi)$ is a linear combination of terms of the form
\[
q(x,y,\tilde M(x,y)^t\xi)a(x,y)\xi^\gamma
\]
for $|\beta|\leq|\alpha|$, where $|\gamma|\leq|\alpha|+|\beta|$. Since $a$ comes from derivatives of $\kappa$, $a\in C^\infty_b(\R^n\times\R^n)$ and then $q\in C^\tau S^{m-|\alpha|-|\beta|-|\gamma|}_{1,0}(\R^n\times\R^n)$ satisfies the transmission condition.\\

Equation \eqref{eq:derivatives pkappa in Ctau} implies that there exist constants $C_{l,\alpha}$ and functions $s_{k,\alpha}$ such that
\begin{align*}
  &\left\|\xi_n^l \sum_{|\beta|\leq N}\left.\dfrac{1}{\alpha!} \partial_{\xi}^{\alpha+\beta}D_y^\beta p_\kappa(\cdot,0,y,\xi)\right|_{y=x}-\sum_{k=-l}^{m-|\alpha|} s_{k,\alpha} (\cdot,\xi')\xi_n^{k+l}\right\|_{C^\tau(\R^{n-1})}  \\
  &\hspace{1cm}\leq C_{l,\alpha}
  \langle\xi'\rangle^{m+l+1-|\alpha|}|\xi_n|^{-1} 
\end{align*}
for $|\xi_n|\geq\langle\xi'\rangle$. This implies that 
\begin{equation*}
  \left\|\xi_n^l \partial_{\xi}^{\alpha} p_L(\cdot,0,\xi)-\sum_{k=-l}^{m-|\alpha|} s_{k,\alpha} (\cdot,\xi')\xi_n^{k+l}\right\|_{C^\tau(\R^{n-1})}
  \leq C_{l,\alpha}
  \langle\xi'\rangle^{m+l+1-|\alpha|}|\xi_n|^{-1}.
\end{equation*}

Now for the derivatives of $p$, we note by \eqref{Dxi'}--\eqref{Dxin} that $D_{\underline{\xi}_n}p$ is directly related to $D_{{\xi}_n}p$, whereas $D_{\underline{\xi}_j}p$ for $j<n$ is a sum of derivatives of $p$, but again the terms can be regrouped to furnish the desired expansions as in \eqref{eq:TransmissionCond}. The norm in $C^\tau(\R^{n-1})$ can be included so it is found altogether that when $p(\cdot,\xi)$ satisfies the transmission condition at $x_n=0$, then so does $p(\cdot,M^t\underline{\xi})$. Derivatives in $\underline{x}$ and $\underline{y}$ can be handled as above including differentiations through $M(x,y)^t\underline{\xi}$. It follows that $p_\kappa:=p_L$ satisfies the transmission condition.
\end{proof}

\section{Nonsmooth Green Operators on Smooth Manifolds}\label{sec:Manifolds}

In this section we follow a similar approach as in \cite[Chapter 7-8]{Wloka}. We will consider Green operators acting on a smooth compact manifold $M$ of dimension $n$.\\

Let $U$ and $V$ be open subsets of $\overline{\R^n_+}$, and let $\kappa:U\to V$ be a diffeomorphism. If $A$ is a pseudodifferential operator with $C^\tau$--coefficients of order $m$, whose kernel has compact support in $U\times U$, then there exists a function $\psi\in C_0^\infty(U)$ such that $A=A(\psi\,\cdot)$. The operator $A:\mathcal{S}(\R^n)\to C^\tau(\R^n)$ can therefore be extended to $C^\infty(U)\to C^\tau(\R^n)$ by setting $Au:=A(\psi\cdot u)$ for all $u\in C^\infty(U)$. We then define the transformed operator $A_\kappa:C^\infty(V)\to C^\tau(\R^n)$ by
\[
 A_\kappa u(y)=
 \begin{cases}
  [A(u\circ\kappa)](\kappa^{-1}(y)) &\text{ if }y\in V,\\
  0 &\text{ if }y\notin V.
 \end{cases}
\]
The operator $A_\kappa$ represents a coordinate change from $U$ to $V$.

\begin{note}
Using this, we can prove Theorem \ref{theorem1} and Theorem \ref{thm: change variables Green} for the case of a bounded smooth diffeomorphism $\kappa$ defined on any open subset of $\overline{\R_+^n}$ on operators whose kernel has compact support.
\end{note}

Remember from Section \ref{section:change of variables Green} that given a bounded smooth diffeomorphism $\kappa:\overline{\R_+^n}\to \overline{\R_+^n}$, we have an induced diffeomorphism $\lambda:\R^{n-1}\to\R^{n-1}$ given by \eqref{eq:lambda}. In the following, given a local chart $\kappa:U\to V\subset\overline{\R_+^n}$ of the manifold $M$, we denote by $\lambda:\partial U\to\partial V$ the corresponding induced diffeomorphism, and by $\nu:=\kappa\times\lambda:U\times\partial U\to V\times\partial V$, where $\partial U:= \partial M\cap U$ and $\partial V:= \partial \overline{\R^n_+}\cap V$.\\
For a linear operator $A=\twobytwo{A_1}{A_2}{A_3}{A_4}:{\begin{matrix}
& C^\infty(U)\\ &\times \\ & C^\infty(\partial U)
\end{matrix}}
\to
{\begin{matrix}
& C^\tau(U)\\ &\times \\ & C^\tau(\partial U),
\end{matrix}}$ we will denote by ${{\nu}}^{-1,*}\circ A\circ\nu^*$ the operator 
\begin{equation}
\label{eq:notation nu-1 A nu}
{{\nu}}^{-1,*}\circ A\circ\nu^*:=\twobytwo{{\kappa^{-1,*}}\circ A_1\circ\kappa^*}{{\kappa}^{-1,*}\circ A_2\circ\lambda^*}{{\lambda}^{-1,*}\circ A_3\circ\kappa^*}{{\lambda}^{-1,*}\circ A_4\circ\lambda^*}.
\end{equation}

Let $A:C^\infty(M)\times C^\infty(\partial M)\to C^\tau(M)\times C^\tau(\partial M)$ be a linear operator, and let $\kappa:U\to V\subset\overline{\R_+^n}$ be some chart for $M$. Let $\nu^*f={(f_1\circ\kappa, f_2\circ \lambda)}$ denote the pull--back of $f{=(f_1,f_2)}\in C^\infty(V)\times C^\infty(\partial V)$ and ${\nu}^{-1,*}g=(g_1\circ\kappa^{-1},g_2\circ\lambda^{-1})$ the push--forward of $g{=(g_1,g_2)}\in C^\infty(U)\times C^\infty(\partial U)$. Let us denote by $i_U:C_0^\infty(U)\times C_0^\infty(\partial U)\to C^\infty(M)\times C^\infty(\partial M)$ the natural embedding and by $r_U:C^\tau(M)\times C^\tau(\partial M)\to C^\tau(U)\times C^\tau(\partial U)$ the restriction operator. Then, we can view $A$ as an operator ${A_U:=}r_U\circ A\circ i_U:C_0^\infty(U)\times C_0^\infty(\partial U)\to C^\tau(U)\times C^\tau(\partial U)$, and the push--forward operator $A_\nu:=\nu^{-1,*}\circ {A_U}\circ\nu^*:C_0^\infty(V)\times C_0^\infty(\partial V)\to C^\tau(V)\times C^\tau(\partial V)$ is the operator $A$ in local coordinates. We would like to consider the push--forward as an operator $A_\nu:C_0^\infty(\overline{\R_+^n})\times C_0^\infty(\R^{n-1})\to C^\tau(\overline{\R_+^n})\times C^\tau(\R^{n-1})$ so that we can use the theory of Green operators on $\overline{\R_+^n}$ given in the previous sections.\\

\begin{defin}
Let $\kappa:U\to V\subset\overline{\R_+^n}$ be a local chart for $M$, and $A:C^\infty(M)\times C^\infty(\partial M)\to C^\tau(M)\times C^\tau(\partial M)$ be a linear map with compact $\supp A\subset U\times \partial U$.{\footnote{{Following \cite[Def.\ 8.4]{Wloka} if $A:C^\infty(M)\to C^\infty(M)$, resp.\ $A:C_0^\infty(M)\to C^\infty(M)$, is a linear map, we define the support of $A$ to be the complement of the largest open set $\mathcal{O}\subset M$ such that for all $\varphi\in C^\infty(M)$, resp.\ $\varphi\in C_0^\infty(M)$,
\begin{enumerate}
\item $A\varphi(x)=0$ if $x\in\mathcal{O}$.
\item $A\varphi\equiv0$ if $\supp(\varphi)\subseteq\mathcal{O}$.
\end{enumerate}
}}} Let $\zeta\in C_0^\infty(U)$ be identically $1$ on a neighborhood of $\supp A$. We define a linear map $A_\nu=T(A;\nu):C_0^\infty(\overline{\R_+^n})\times C_0^\infty(\R^{n-1})\to  C^\tau(\overline{\R_+^n})\times C^\tau(\R^{n-1})$, called the push--forward of $A$ by $\nu$, as follows: for $f\in C_0^\infty(\overline{\R_+^n})\times C^\infty_0(\R^{n-1})$,
\begin{align*}
 T(A;\nu)f&:=[\zeta\cdot A(\zeta\cdot(f\circ\nu))]\circ\nu^{-1} =\nu^{-1,*}\circ\zeta\cdot A\circ\nu^*(\zeta\circ\nu^{-1})\cdot f,
\end{align*}
and extended by zero outside $\nu(U\times\partial U):=\kappa(U)\times\lambda(\partial U)=V\times\partial V$.
\end{defin}

The support of $T(A;\nu)$ is $\nu(\supp A)$, a compact subset of $\nu(U\times\partial U)=V\times\partial V$. Conversely, if $S:C_0^\infty(\overline{\R_+^n})\times C_0^\infty(\R^{n-1})\to C^\tau(\overline{\R_+^n})\times C^\tau(\R^{n-1})$ has compact support in $\nu(U\times\partial U)$, then $S=T(A;\nu)$ for some $A:C^\infty(M)\times C^\infty(\partial M)\to C^\tau(M)\times C^\tau(\partial M)$ with support in $U\times \partial U$. In fact the push--forward $A=T(S;\nu^{-1}):C^\infty(M)\times C^\infty(\partial M)\to C^\tau(M)\times C^\tau(\partial M)$ having support in $U\times \partial U$ is defined by
\begin{align*}
 T(S;\nu^{-1})f&:=[(\zeta\circ\nu^{-1})\cdot S((\zeta\circ\nu^{-1})\cdot(f\circ\nu^{-1}))]\circ\nu\\
 &=\nu^*\circ (\zeta\circ\nu^{-1})\cdot S\circ \nu^{-1,*}\zeta\cdot f.
\end{align*}
The push--forwards $A\mapsto T(A;\nu)$ and $S\mapsto T(S;\nu^{-1})$ also called push--forward and pull--back respectively, are mutually inverse one to one maps between linear maps $A:C^\infty(M)\times C^\infty(\partial M)\to C^\tau(M)\times C^\tau(\partial M)$ with compact support in $U\times \partial U$ and linear maps $S:C_0^\infty(\overline{\R_+^n})\times C^\infty(\R^{n-1})\to C^\tau(\overline{\R_+^n})\times C^\tau(\R^{n-1})$ with compact support in $\nu(U\times\partial U)$.\\

Similar to the definition of pseudodifferential operators on a manifold as in \cite[Definition 8.7]{Wloka}, we define Green operators with $C^\tau$--coefficients on a manifold. We will continue using the notation given in Equation \eqref{eq:notation nu-1 A nu} and additionally, if 
\begin{equation*}
A=\twobytwo{A_1}{A_2}{A_3}{A_4}:{\begin{matrix}
& C^\infty(U)\\ &\times \\ & C^\infty(\partial U)
\end{matrix}}
\to
{\begin{matrix}
& C^\tau(U)\\ &\times \\ & C^\tau(\partial U),
\end{matrix}}  
\end{equation*}
 is a linear operator defined on $U\subseteq\overline{\R_+^n}$, given two functions $\varphi,\psi\in C_0^\infty(\overline{\R_+^n})$, we denote by $\varphi A\psi$ the operator
\begin{equation}
\label{eq:notation phi A psi}
\varphi A\psi:=\twobytwo{({\varphi|_U})A_1({\psi|_U})}{({\varphi|_U}) A_2({\psi|_{\partial U}})}{({\varphi|_{\partial U}}) A_3({\psi|_U})}{({\varphi|_{\partial U}}) A_4({\psi|_{\partial U}})}.
\end{equation}

\begin{defin}
\label{def:green op in a mfld}
 Let $M$ be a compact $C^\infty$--manifold with boundary. A linear operator $A:C^\infty(M)\times C^\infty(\partial M)\to C^\tau(M)\times C^\tau(\partial M)$ is called a Green operator with $C^\tau$--coefficients on $M$ of order $m\in\R$, if for every coordinate chart $(U,\kappa)$ on $M$, and for all $\varphi,\psi\in C_0^\infty(U)$, the push--forward operator
 \[
 T(\varphi A\psi;\nu)= \nu^{-1,*}(\varphi A\psi)\nu^*=(\varphi\circ\nu^{-1})T(A;\nu)(\psi\circ\nu^{-1}),
 \]
 is a Green operator with $C^\tau$--coefficients on $\overline{\R_+^n}$ of order $m$.\\
\end{defin}
We note that, if $(U,\kappa)$, $(U',\kappa')$ are two coordinate charts with $U\cap U'=\emptyset$, then $(U\cup U',\kappa\cup \kappa')$ is another coordinate chart, where 
\[
 (\kappa\cup\kappa')(x):=\begin{cases}
                         \kappa(x) & \text{ if }x\in U,\\
                         \kappa'(x) & \text{ if }x\in U',\\	      
                        \end{cases}
\]
and without loss of generality we take $\kappa$ and $\kappa'$ to have disjoint image so that $\kappa\cup\kappa'$ has an inverse.

\begin{note}
 If both $\kappa$ and $\kappa'$ are bounded smooth diffeomorphisms, then $\kappa\cup\kappa'$ is also a bounded smooth diffeomorphism.
\end{note}

\begin{defin}
\label{def:Green op mfld CtauCinfty kernel}
 The Green operator $A$ has $C^\tau - C^\infty$--kernel, if for every coordinate chart $(U,\kappa)$ on $M$, and for all $\varphi,\psi\in C_0^\infty(U)$, the push--forward operator $\nu^{-1,*}(\varphi A\psi)\nu^*$ is a Green operator on $\overline{\R_+^n}$ with $C^\tau - C^\infty$--kernel as in Definition \ref{def:Green op CtauCinfty kernel}.
\end{defin}

Similarly to \cite[Proposition 8.8]{Wloka}, we have that Green operators with $C^\tau$--coefficients on $M$ are quasi-local:
\begin{prop}
\label{prop:phi A psi Green op mfld}
 Let $\varphi,\psi\in C^\infty(M)$ be such that $\supp\varphi\cap\supp\psi=\emptyset$ and let $A$ be a Green operator with $C^\tau$--coefficients on $M$. Then $\varphi A\psi$ is a Green operator with $C^\tau - C^\infty$--kernel.
\end{prop}

\begin{proof}
 Since the supports of $\varphi$ and $\psi$ are compact and disjoint from each other,
\begin{align*}
 \supp\varphi&\subset\mathcal{U}:=U_1\cup\ldots\cup U_p,\qquad
 \supp\psi\subset\mathcal{V}:=U_{p+1}\cup\ldots\cup U_N,
\end{align*}
where $(U_j,\kappa_j)$, $j=1,\ldots,N$, are some coordinate charts on $M$ and $\mathcal{U}\cap\mathcal{V}=\emptyset$. We choose an open set ${W}$ such that 
\[
 \supp\varphi\cup\supp\psi\subset{W}\subset\overline{{W}}\subset\mathcal{U}\cup\mathcal{V}.
\]
Then $U_1,\ldots,U_N,\overline{{W}}^c$ is an open cover of $M$, and we choose a subordinated partition of unity $\{\phi_j\}_{1\leq j\leq N+1}$, that is, $\sum_{j=1}^{N+1}\phi_j=1$, $\supp\phi_j\subset U_j$, for all $j=1,\ldots,N$ and $\supp\phi_{N+1}\cap{W}=\emptyset$.\\
Then 
\begin{equation*}
 \varphi A\psi=\sum_{j,k}\phi_j\varphi A\psi\phi_k 
\end{equation*}
 where the sum is taken over the indices $j=1,\ldots,p$, and $k=p+1,\ldots,N$, since these are the only indices for which $\phi_j\varphi$ and $\psi\phi_k$ do not necessarily vanish.\\
Now choose functions $\chi_j\in C_0^\infty(U_j)$ such that $\chi_j\equiv1$ on $\supp\phi_j$. Then
\begin{equation}
 \label{eq:varphi A psi partition of unity}
 \varphi A\psi=\sum_{j,k}\chi_j\varphi(\phi_j A\phi_k)\psi\chi_k.
\end{equation}
Joining the two coordinate charts $(U_j,\kappa_j)$ and $(U_k,\kappa_k)$ as above, we obtain a coordinate chart $(U_{jk},\kappa_{jk}):=(U_j\cup U_k,\kappa_j\cup\kappa_k)$, and then by Definition \ref{def:green op in a mfld} the operator ${\nu_{jk}}^{-1,*}(\phi_j A\phi_k)\nu_{jk}^*$ is a Green operator with $C^\tau$--coefficients on $\overline{\R_+^n}$ of order $m$, where as before $\nu_{jk}:=\kappa_{jk}\times\lambda_{jk}$. Since $\supp(\chi_j\varphi)\cap\supp(\psi\chi_k)=\emptyset$, Propositions \ref{prop:kernel phi p psi}, \ref{prop:phi k(x,Dx)psi CtauCinfty kernel Poisson}, \ref{prop:phi k(x,Dx)psi CtauCinfty kernel trace}, \ref{prop:phi k(x,Dx)psi CtauCinfty kernel sgo}, imply that the push--forward operator ${\nu_{jk}}^{-1,*}(\chi_j\varphi(\phi_j A\phi_k)\psi\chi_k)\nu_{jk}^*$ of each term in the sum \eqref{eq:varphi A psi partition of unity} has a $C^\tau-C^\infty$--kernel. \\ Hence, if $(U,\kappa)$ is an arbitrary chart and $\tilde{\varphi},\tilde{\psi}\in C^\infty_0(U)$, 
$$
\nu^{-1,*}{\nu_{jk}}^{-1,*}(\chi_j\tilde{\varphi}\varphi(\phi_j A\phi_k)\psi\tilde{\psi}\chi_k)\nu_{jk}^*\nu^{*}
$$ 
has a $C^\tau-C^\infty$--kernel as well. Therefore $\nu^{-1,*}\tilde{\varphi}\varphi A\psi\tilde{\psi}\nu^{*}$ has a $C^\tau - C^\infty$--kernel. Since $(U,\kappa)$ and $\tilde{\varphi},\tilde{\psi}\in C^\infty_0(U)$ have been arbitrary, $\varphi A\psi$ has a $C^\tau-C^\infty$--kernel.
\end{proof}

Our main result is as follows:
\begin{thm}
Let $\mathcal{U}$ be an atlas on $M$. A linear operator $A$ is a Green operator with $C^\tau$--coefficients on $M$ of order $m$ if and only if the following conditions are satisfied:
\begin{enumerate}
 \item[(i)] For every chart $(U,\kappa)\in\mathcal{U}$ and for all $\varphi,\psi\in C_0^\infty(U)$, the operator $\nu^{-1,*}(\varphi A\psi)\nu^*$ is a Green operator with $C^\tau$--coefficients on $\overline{\R_+^n}$ of order $m$.
 \item[(ii)] For all $\varphi,\psi\in C^\infty(M)$ with $\supp\varphi\cap\supp\psi=\emptyset$, the operator $\varphi A\psi$ has a $C^\tau - C^\infty$--kernel.
\end{enumerate}
\end{thm}

\begin{proof}
 The necessity follows from Definition \ref{def:green op in a mfld} and Proposition \ref{prop:phi A psi Green op mfld}. For the sufficiency, choose a partition of unity $\{\varphi_j:{j\in J}\}$ subordinated to the atlas $\mathcal{U}${, i.e., for every $j\in J$ there is some $(U_j,\kappa_j)\in \mathcal{U}$ such that $\supp\varphi_j \subseteq U_j$,} and corresponding functions $\psi_j\in C_0^\infty(U_j)$ such that $\psi_j\equiv1$ on the support of $\varphi_j$.

Then we have
\[
 A\ \cdot=\sum_j\varphi_j A\ \cdot=\sum_j\varphi_j A[\psi_j\ \cdot\ ]+\sum_j\varphi_j A[(1-\psi_j)\ \cdot\ ].
\]
Let $\varphi,\psi\in C_0^\infty(U_j)$. Since $\supp\varphi_j\cap\supp(1-\psi_j)=\emptyset$, by Condition $(ii)$ the terms $\varphi_j A[(1-\psi_j)\cdot]$ have $C^\tau - C^\infty$--kernels. After composing them with the multiplication operators by $\varphi$ and $\psi$,  by Definition \ref{def:Green op mfld CtauCinfty kernel} the push--forward of such terms are Green operators on $\overline{\R_+^n}$ with $C^\tau - C^\infty$--kernel, which has compact support in $\overline{\R_+^n}\times\R^{n-1}\times\overline{\R_+^n}\times\R^{n-1}$. Hence by Lemma \ref{lem:green op order -infty has CtauCinfty kernel} and Definition \ref{def:green op in a mfld} the operators $\varphi\varphi_j A[(1-\psi_j)\psi\cdot]$ are Green operators with $C^\tau$--coefficients on $M$ of order $-\infty$.\\
Let $(U,\kappa)$ be an arbitrary chart. We can assume without loss of generality that $\kappa\circ\kappa_j^{-1}|_{A_j}$, where $A_j:= \kappa_j(\supp \psi_j)$, 
 can be extended to a bounded smooth diffeomorphism from $\overline{\R_+^n}$ to $\overline{\R_+^n}$. Otherwise we replace $\varphi_j$ and $\psi_j$ by some $\varphi_{j,1},\ldots,\varphi_{j,N}, \psi_{j,1},\ldots,\psi_{j,N}\in C_0^\infty(U_j)$ (with sufficiently small support) such that $\varphi_j = \sum_{k=1}^N \varphi_{j,k}$ and $\psi_{j,k}\equiv 1$ on $\supp \varphi_{j,k}$  and the previous condition is valid for $\varphi_{j,k}$.
We will show that for any $\varphi,\psi\in C_0^\infty(U)$, the push--forward operators $\big(\varphi\varphi_j A\psi_j\psi\big)_\nu:=\nu^{-1,*}\varphi\varphi_j A\psi_j\psi\nu^*$ are Green operators with $C^\tau$--coefficients on $\overline{\R_+^n}$ of order $m$. We have
\begin{equation}
\label{eq:operator quasilocal}
 \big(\varphi\varphi_j A\psi_j\psi\big)_\nu=\big[\big(\varphi\varphi_j A\psi_j\psi\big)_{\nu_j}\big]_{\nu\circ\nu_j^{-1}}.
 \end{equation}
By Condition $(i)$ we have that $\big(\varphi\varphi_j A\psi_j\psi\big)_{\nu_j}$ is a Green operator with $C^\tau$--coefficients on $\overline{\R_+^n}$ of order $m$. 
Therefore by Theorem \ref{theorem1} and Theorem \ref{thm: change variables Green}, the operator \eqref{eq:operator quasilocal} is a Green operator with $C^\tau$--coefficients on $\overline{\R_+^n}$ of order $m$, as we wanted to show.
\end{proof}

\def\cprime{$'$} \def\ocirc#1{\ifmmode\setbox0=\hbox{$#1$}\dimen0=\ht0
  \advance\dimen0 by1pt\rlap{\hbox to\wd0{\hss\raise\dimen0
  \hbox{\hskip.2em$\scriptscriptstyle\circ$}\hss}}#1\else {\accent"17 #1}\fi}



  

\end{document}